\documentclass[12pt]{amsart}
\usepackage{amsfonts, amsmath, amsthm, amssymb}

\usepackage{graphicx}
\usepackage{epstopdf}
\usepackage{psfrag}

\newtheorem{pro}{Proposition}[section]
\newtheorem{thm}[pro]{Theorem}
\newtheorem{lem}[pro]{Lemma}
\newtheorem{clm}[pro]{Claim}

\theoremstyle{definition}
\newtheorem{dfn}[pro]{Definition}

\newtheorem{rmk}[pro]{Remark}
\newcommand{\VV}{\mathcal V}
\newcommand{\WW}{\mathcal W}
\newcommand{\CC}{\mathcal C}

\newcommand{\bdy}{\partial}
\newcommand{\BB}{\mathcal B}
\newcommand{\EE}{\mathcal E}
\newcommand{\DD}{\mathcal D}

\newcommand{\TT}{\mathcal T}
\newcommand{\KK}{\mathcal K}
\renewcommand{\slash}{/ \hspace{-.05in} /}
\newcommand{\ind}{{\rm ind}}

\newcommand{\plex}[1]{\ensuremath{[{#1}]}}

\newsavebox{\savepar}
\newenvironment{boxit3}
	{\begin{lrbox}{\savepar}\begin{minipage}[b]{3in}}
        {\end{minipage}\end{lrbox}\fbox{\usebox{\savepar}}}

\newenvironment{boxit2}
	{\begin{lrbox}{\savepar}\begin{minipage}[b]{2in}}
        {\end{minipage}\end{lrbox}\fbox{\usebox{\savepar}}}

\newenvironment{boxit2.5}
	{\begin{lrbox}{\savepar}\begin{minipage}[b]{2.5in}}
        {\end{minipage}\end{lrbox}\fbox{\usebox{\savepar}}}

\title{Normalizing Topologically Minimal Surfaces I: Global to Local Index}
\date{\today}
\author{David Bachman}

\begin{document}

\begin{abstract}
We show that in any triangulated 3-manifold, every index $n$ topologically minimal surface can be transformed to a surface which has local indices (as computed in each tetrahedron) that sum to at most $n$. This generalizes classical theorems of Kneser and Haken, and more recent theorems of Rubinstein and Stocking, and is the first step in a program to show that every topologically minimal surface has a normal form with respect to any triangulation.
\end{abstract}
\maketitle

\section{Introduction}

\markright{NORMALIZING TOPOLOGICALLY MINIMAL SURFACES I}

Minimal surfaces are classical objects of geometry, defined by the condition that their mean curvature is zero everywhere.  Equivalently, minimal surfaces represent critical values for surface area, among the space of all surfaces embedded in a given 3-manifold. The word {\it minimal} is a bit of a misnomer, since such critical values are not necessarily local minima. 

Recently, the author introduced {\it topologically minimal} surfaces \cite{TopIndexI}. Although their definition is quite different, these surfaces possess many of the same properties as the geometrically minimal surfaces described above. For example, it is well known that a geometrically minimal surface and an incompressible surface can be isotoped so that their intersection curves are essential on both. This has also been shown to be true of topologically minimal surfaces \cite{TopIndexI}. 

In recent work,  Colding and Minicozzi \cite{cm1}, \cite{cm2}, \cite{cm3}, \cite{cm4} have classified the possible local pictures of a geometrically minimal surface. For example, given a ball $B$ and a geometrically minimal surface $H$, Colding and Minicozzi showed that if $H \cap B$ is simply connected, then it is either the graph of a function, or a helicoid. In \cite{TopIndexI}, the author conjectured that every topologically minimal surface can be isotoped to be geometrically minimal. An immediate consequence would be that every topologically minimal surface can be isotoped so that its possible local pictures are described by the results of Colding and Minicozzi. 

This is the first in a series of papers where we study the way in which a topologically minimal surface $H$ can be isotoped with respect to the tetrahedra of a triangulation. Our goal is to show that $H$ can be isotoped so that it meets each tetrahedron in precisely the same way that a geometrically minimal surface can meet a ball, as described by Colding and Minicozzi. In such a position, for example, if $\Delta$ is a tetrahedron and $H \cap \Delta$ is simply connected, then we show that $H \cap \Delta$ is contained in either a plane or a helicoid. 

\subsection{Generalizing normal and almost normal surfaces.\\} 

The present work can be viewed from two different perspectives. As described above, it can be thought of as further strengthening the tie between geometrically and topologically minimal surfaces. Alternately, this work can be seen as a natural extension of results of Kneser, Haken, Rubinstein and Stocking on normal and almost normal surfaces.  Just as in the case of geometrically minimal surfaces, topologically minimal surfaces have a well defined {\it index}. Index 0 surfaces are precisely the incompressible surfaces of Haken \cite{haken:61}, and index 1 surfaces are the strongly irreducible surfaces of Casson and Gordon \cite{cg:87}. 

In 1929, Kneser \cite{kneser:29} showed that an essential sphere $S$ can be isotoped to a normal form with respect to an arbitrary triangulation. In this position, $S$ meets each tetrahedron $\Delta$ in the same way a plane might. In particular, $S \cap \Delta$ is a collection of triangles and quadrilaterals, called {\it normal disks}. In later work, Haken showed that incompressible (i.e.~index 0) surfaces can be isotoped into the same normal form \cite{haken:68}. In \cite{rubinstein:93} and \cite{stocking:96}, Rubinstein and Stocking independently showed that strongly irreducible (i.e.~index 1) surfaces can be isotoped to meet each tetrahedron of a given triangulation in a collection of normal disks, with at most one possible exception in one tetrahedron. This exceptional piece is either a saddle-shaped octagon, or two normal disks connected by an unknotted tube (similar to a {\it catenoid}, a classical geometrically minimal surface). 

In \cite{crit}, the author introduced {\it critical surfaces}, which are now recognized as the topologically minimal surfaces with index 2 \cite{TopIndexI}. In unpublished work, both the author \cite{nlizg} and Johnson \cite{johnson3} claimed that critical surfaces can be isotoped to meet the tetrahedra of a triangulation in a collection of planar pieces, with at most two exceptions. If there are exactly two exceptional pieces, then each is either a tube or an octagon, as described above. If there is a single exceptional piece then it may be a tube or an octagon as well, or something a bit more complicated. The possibilities are a helical shaped 12-gon, three normal disks connected by two unknotted tubes, an octagon tubed to a normal disk, or an octagon tubed to itself.

In this paper we give a local definition of the index of a surface. This is an index that is computed separately in each tetrahedron. It follows from this definition that the surfaces in a tetrahedron that have index 0 are the normal disks, and the surfaces with index 1 are tubes and octagons. In \cite{TopMinNormalIII} we show that the surfaces in a tetrahedron with index 2 are the other types of exceptional pieces mentioned above. Using this language, we can concisely summarize all of the results mentioned above as follows:

\begin{thm} \cite{kneser:29}, \cite{haken:68}, \cite{rubinstein:93}, \cite{stocking:96}, \cite{nlizg}, \cite{johnson3}
\label{t:classical}
A closed, topologically minimal surface $H$ whose index is $n \le 2$ in a triangulated 3-manifold can be isotoped so that the sum, taken over all tetrahedra $\Delta$, of the indices of $H \cap \Delta$ is at most $n$. 
\end{thm}

The main result of this paper is a generalization of Theorem \ref{t:classical} to arbitrary values of the index $n$, and to surfaces with possibly non-empty boundary. In particular, we show:

\begin{thm}
\label{t:main_paraphrase}
A properly embedded, topologically minimal surface $H$ whose index is $n$ in a triangulated 3-manifold can be compressed, $\bdy$-compressed and isotoped to a surface $H'$ so that the sum, taken over all tetrahedra $\Delta$, of the indices of $H' \cap \Delta$ is at most $n$. 
\end{thm}

The compressions in the above theorem are very restricted, and only happen when the initial index $n$ is at least three. In particular, when $H$ is closed and $n \le 2$, then we recover Theorem \ref{t:classical}.

Theorem \ref{t:main_paraphrase} is the first step in producing a normal form for any index $n$ topologically minimal surface, with respect to an arbitrary triangulation. To complete the picture of such a normal form, one would have to classify all possible surfaces {\it in a single tetrahedron} whose index is $n$. In the first sequel to this paper \cite{TopMinNormalII} , we begin such a classification. We show there that if a topologically minimal surface (with index at least 1) in a tetrahedron is simply connected then it is a helicoid. Furthermore, the number of ``turns" of such a helicoid is directly proportional to its index. This, combined with Theorem \ref{t:main_paraphrase}, gives us the desired topological version of the Colding-Minicozzi result mentioned above: any topologically minimal surface in a triangulated 3-manifold can be transformed so that if the intersection with a given tetrahedron is simply connected, then it is either contained in a plane or a helicoid. 

In the second sequel to this paper \cite{TopMinNormalIII} we classify all index 2, non-simply connected surfaces in a tetrahedron. This result, combined with Theorem \ref{t:main_paraphrase} and the results from \cite{TopMinNormalII}, gives a normal form for all index 2 topologically minimal surfaces in a 3-manifold, with respect to an arbitrary triangulation. Having such a normal form has several known and potential applications. For example, in \cite{TopMinNormalIII} we combine this normalization result with Corollary 11.2 of \cite{Stabilizing} to show that there exists a single 3-manifold with infinitely many Heegaard splittings with Hempel distance (see \cite{hempel:01}) exactly one, a fact that was not previously known. In joint work with Ryan Derby-Talbot and Eric Sedgwick, we use the normalization of index 2 surfaces to show that the set of Heegaard surfaces of 3-manifolds with toroidal boundary does not change after ``generic" Dehn filling \cite{HeegaardDehn}.  A potential application is an algorithm to determine if two arbitrary Heegaard surfaces are isotopic, and if they are not, the number of stabilizations required to make them isotopic. A second potential application is an algorithm to determine if a given Heegaard surface is strongly irreducible. 

The author thanks Ryan Derby-Talbot and Eric Sedgwick for countless helpful conversations about the present work. Much of the difficulty in working in this area is in finding good notation for the myriad of types of disks and complexes that arise. Most of the notation used here was developed jointly with Derby-Talbot and Sedgwick for the paper \cite{Index1Normal}, where we present a new proof of the normalization of index 1 surfaces. Many similar ideas are used in both that proof and the present paper.

\section{Compressing Disks and The Disk Complex}

Throughout this paper $H$ will represent a connected, orientable surface that is properly embedded in a compact, orientable, irreducible 3-manifold $M$ with incompressible boundary and triangulation $\TT$. Furthermore, we will assume $H$ separates $M$ and is not contained in a ball. In this section we give basic definitions and facts regarding compressing disks and associated disk complexes.  

\begin{dfn}
An embedded loop or arc $\alpha$ on $H$ is said to be {\it inessential} if it cuts off a subdisk of $H$, and {\it essential} otherwise.
\end{dfn}

We define three types of compressing disks, extending the standard definitions to account for the presence of a given complex $\KK$.   We are concerned here with the cases when $\KK = \emptyset$, $\TT^1$, or $\TT^2$, the $1$- and $2$-skeleton of our triangulation $\TT$. However, the case when $\KK$ is a knot or link in a 3-manifold may be of interest to those who study ``thin position" and related 3-manifold techniques. In what follows, $N(\KK)$ refers to a regular neighborhood of $\KK$ in $M$.

\begin{dfn}
A {\it compressing disk} $C$ for $(H,\KK)$ is a disk embedded in $M-N(\KK)$ so that $\alpha = C \cap H=\bdy C$ is an essential loop in $H-N(\KK)$.

Let $\CC(H,\KK)$ denote the set of  compressing disks for $(H,\KK)$.
\end{dfn}

\begin{dfn}
A {\it $\bdy$-compressing disk}  $B$ for $(H,\KK)$ is a disk embedded in $M-N(\KK)$ so that $\bdy B=\alpha \cup \beta$, where $\alpha=B \cap H$ is an essential arc on $H-N(\KK)$ and $B \cap \bdy M=\beta$. A $\bdy$-compressing disk is {\it fake} if the arc $\beta$ is parallel into $\bdy H$ on $\bdy M-N(\KK)$ and {\it real} otherwise. 

Let $\BB(H,\KK)$ denote the set of $\bdy$-compressing disks for $(H,\KK)$, and $\BB^r (H,\KK)$ the real $\bdy$-compressing disks. 
\end{dfn}

\begin{dfn} 
When $\KK=\TT^1$ or $\TT^2$, we define an {\it edge-compressing disk} $E$ for $(H,\KK)$ to be a disk embedded in $M$ so that $\bdy E =\alpha \cup \beta$, where $\alpha= E  \cap H$ is an arc on $H$ and $\beta = E \cap \KK \subset e$, for some edge $e$ of $\TT^1$. We say $E$ is a {\it boundary} edge-compressing disk when the edge of $T^1$ that it is incident to is contained in $\bdy M$, and is an {\it interior} edge-compressing disk otherwise. 

Let $\EE(H,\KK)$ denote the set of edge compressing disks for $(H,\KK)$.
\end{dfn}

Notice that each element of $\EE(H,\TT^1)$ meets $M-N(\TT^1)$ in a real $\bdy$-compressing disk for $(H-N(\TT^1),\emptyset)$. Hence, $\EE(H,\TT^1)$ can be identified with a subset of $\BB^r(H-N(\TT^1), \emptyset)$. This fact will become important in Section \ref{s:Stage3}. 

\begin{dfn}
Henceforth, we will employ the following shortened notation:
	\begin{itemize}
		\item $\CC(H)=\CC(H,\emptyset)$
		\item $\BB^r(H)=\BB^r(H,\emptyset)$
		\item $\CC \BB^r(H,\KK)=\CC(H,\KK) \cup \BB^r(H,\KK)$
		\item $\CC\EE(H,\KK)=\CC(H,\KK) \cup \EE(H,\KK)$
		\item $\DD(H,\KK)=$ any one of $\CC(H,\KK)$, $\BB^r(H,\KK)$, $\EE(H,\KK)$, $\CC \BB^r(H,\KK)$ or $\CC\EE(H,\KK)$. 
	\end{itemize}
\end{dfn}

\begin{dfn}
\label{d:Compression}
Suppose $D \in \DD(H,\KK)$.  We construct a surface $H/D$, which is said to have been obtained by {\it surgering}  along $D$, as follows. Let $M(H,\KK)$ denote the manifold obtained from $M-N(\KK)$ by cutting open along $H$. Let $N(D)$ denote a neighborhood of $D$ in $M(H,\KK)$. Construct the surface $H'$ in $M-N(\KK)$ by removing $N(D) \cap H$ from $H$ and replacing it with the frontier of $N(D)$ in $M(H,\KK) $. The surface $H/D$ is then obtained from $H'$ by discarding any component that lies in a ball in $M$. See, for example, Figure \ref{f:HcompE}.
\end{dfn}

\begin{figure}
\psfrag{D}{$E$}
\psfrag{H}{$H$}
\psfrag{T}{$\TT^1$}
\psfrag{t}{$N(\TT^1)$}
\psfrag{h}{$H/E$}
\[\includegraphics[width=4in]{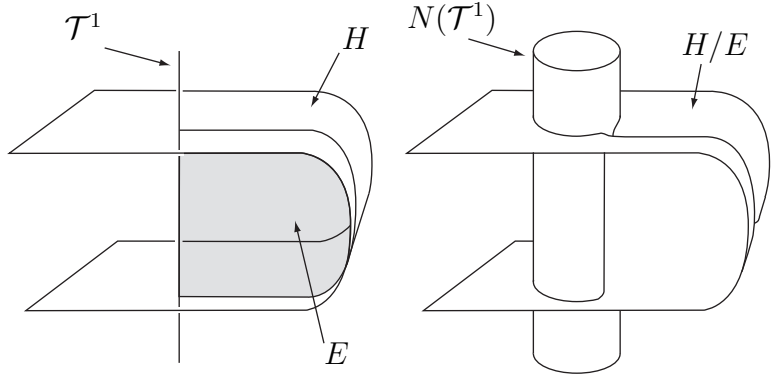}\]
\caption{Surgering along an edge-compressing disk $E \in \EE(H,\TT^1)$ produces a surface in $M-N(\TT^1)$.}
\label{f:HcompE}
\end{figure}

\begin{dfn}
We say $H$ is {\it compressible} if $\CC(H)$ is non-empty, and {\it incompressible} otherwise.  If $D$ is a compressing disk for $H$ then $H/D$ is said to have been obtained from $H$ by {\it compressing along $D$}.
\end{dfn}

\begin{dfn}
\label{d:H/tau}
Suppose $\DD$ is a collection of elements of $\DD(H,\KK)$ whose intersections with $M-N(\KK)$ are pairwise disjoint. Then we form the surface $H/\DD$ by simultaneously surgering along each. 
\end{dfn}

\begin{lem}
The previous definition is well defined, up to isotopy. 
\end{lem}

\begin{proof}
If $\DD$ meets $H$ in a collection of pairwise non-parallel curves, then this collection is unique, and thus $H/\DD$ is unique up to isotopy. Suppose, then, that two disks, $D$ and $E$, in $\DD$ meet $H$ in parallel curves. Then $D$ and $E$ can potentially be  isotoped to meet $H$ in two different arrangements. Surgery of $H$ along each such arrangement produces surfaces $H_1$ and $H_2$ that differ only by a sphere. By the assumed irreducibility of $M$, this sphere bounds a ball. To form $H/DE$ we remove this sphere from either $H_1$ or $H_2$. In either case we get the same surface (up to isotopy), and thus the operation is well-defined.

The only other case that needs to be considered is when $D$ and $E$ are both $\bdy$-compressions. In this case the surfaces $H_1$ and $H_2$ described above differ by a disk with boundary on $\bdy M$. By the assumed incompressibility of $\bdy M$, and the irreducibility of $M$, this disk lies in a ball in $M$. It is thus removed from either $H_1$ or $H_2$ to form $H/DE$. in either case, the surface thus obtained is the same, up to isotopy. 
\end{proof}

Suppose now $H$ is a properly embedded surface in $M$, transverse to $\TT^1$, and  $\EE$ is a collection of interior edge-compressing disks for $(H,\TT^1)$. Then the surface $H/\EE$ is defined only to be a surface in $M-N(\TT^1)$. We now give a related operation that produces a surface $H \slash \EE$ which is properly embedded in $M$ and transverse to $T^1$. The surfaces $H$ and $H\slash \EE$ may be isotopic in $M$, as in Figure \ref{f:surgery}, or $H \slash \EE$ may be obtained from $H$ by surgery along some number of compressing disks, as in Figure \ref{f:surgery2}.

\begin{figure}
\psfrag{H}{$H$}
\psfrag{D}{$H \slash \DD$}
\[\includegraphics[width=4 in]{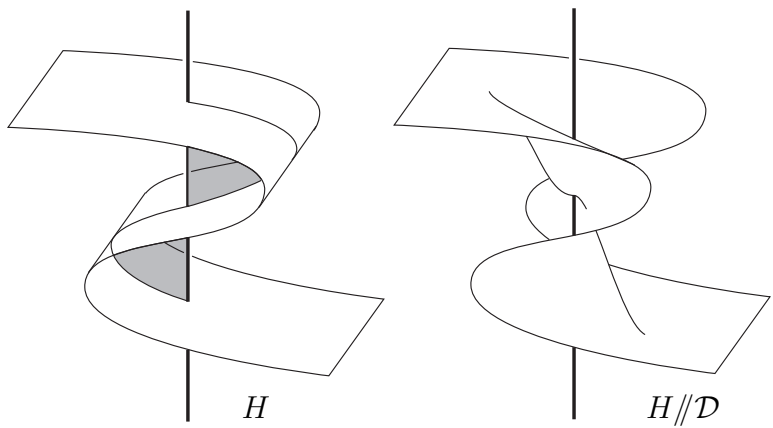}\]
\caption{The collection $\EE$ consists of two disks that meet at a point of $\TT^1$, but are disjoint in $M-N(\TT^1)$. The surfaces $H$ and $H \slash \EE$ are isotopic in $M$.}
\label{f:surgery}
\end{figure}

\begin{figure}
\psfrag{H}{$H$}
\psfrag{D}{$H \slash \DD$}
\[\includegraphics[width=4 in]{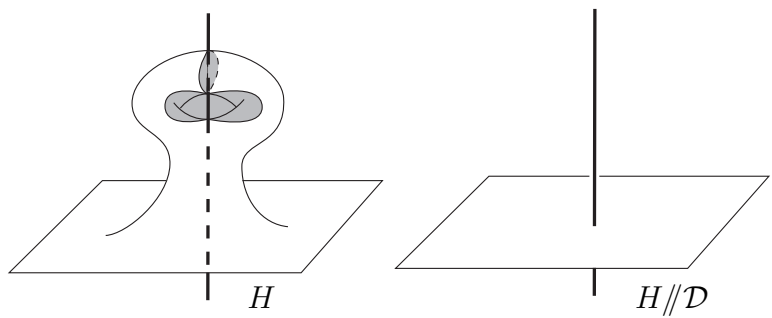}\]
\caption{The collection $\EE$ consists of four disks that all meet at a point of $\TT^1$. The surface $H \slash \EE$ can be obtained from $H$ by surgery along a single compressing disk.}
\label{f:surgery2}
\end{figure}

\begin{dfn}
Suppose $\EE$ is a collection of interior edge-compressing disks for $(H,\TT^1)$ which are pairwise disjoint in the complement of $N(\TT^1)$. If, for each interior edge $e$ of $\TT^1$, the surface $H/\EE$ meets $\bdy N(e)$ in essential loops, then we cap off all such loops with disks in $N(\TT^1)$ to form a properly embedded surface $H \slash \EE$ in $M$ that is transverse to $\TT^1$.
\end{dfn}

We define an equivalence relation on disks as follows:

\begin{dfn}
Disks $C,C' \in \DD(H,\KK)$ are {\it equivalent}, $C\sim C'$, if $C$ and $C'$ are isotopic in $M-N(\KK)$, through disks in $\DD(H,\KK)$. 
\end{dfn}

\begin{dfn}
\label{d:DiskComplex}
The {\it disk complex} $\plex{\DD(H,\KK)}$ is the complex defined as follows: vertices correspond to equivalence classes $[D]$ of the disk set $\DD(H,\KK)$. A collection of $n$ vertices spans an $(n-1)$-simplex if there are representatives of each whose intersections with $M-N(\KK)$ are pairwise disjoint.
\end{dfn}

\begin{dfn}
\label{d:H/tau}
Suppose $\tau$ is a simplex of $\plex{\DD(H,\KK)}$. Then $H/\tau=H/\DD$, where $\DD$ is a pairwise disjoint collection of representatives of the vertices of $\tau$.
\end{dfn}

\begin{dfn}
\label{d:Indexn}
The {\it homotopy index} of a complex $\plex{\DD(H,\KK)}$, denoted $\ind\plex{\DD(H,\KK)}$, is defined to be 0 if $\plex{\DD(H,\KK)}=\emptyset$, and the smallest $n$ such that $\pi_{n-1}(\plex{\DD(H,\KK)})$ is non-trivial, otherwise. (If a complex is contractible, its homotopy index is left undefined.)
\end{dfn}

\begin{dfn}
\label{d:TopMin}
\cite{TopIndexI} A surface $H$ in a 3-manifold is {\it topologically minimal} if $\plex{\CC(H)}$ is either empty or non-contractible. When $H$ is topologically minimal, we say the {\it index of $H$} is the number $\ind\plex{\CC(H)}$.
\end{dfn}

We now come to our main theorem, which guarantees a process by which we can begin with an index $n$, topologically minimal surface $H_0$, and end up with a surface $H$ where $\ind\plex{\CC \EE(H,\TT^2)} \le n$.

\begin{thm}
\label{t:main1}
Suppose $H_0$ is a topologically minimal, non-peripheral surface in $M$ whose index is $n$. Then there is a surface $H$, obtained from $H_0$ by the sequence of $\bdy$-compressions, interior edge-compressions, and isotopies given in the flow chart of Figure \ref{f:flowchart}, such that \[\ind\plex{\CC \EE(H,\TT^2)} \le n.\]
\end{thm}

\begin{figure}
\begin{center}
\psfrag{A}{\begin{boxit2}\tiny Begin with $H_0$ s.t. $\ind\plex{\CC(H_0)}$ is well-defined. Initially define $H=H_0$.\end{boxit2}}
\psfrag{B}{\tiny Is $\ind\plex{\CC \BB^r(H)} \le \ind\plex{\CC(H)}$?}
\psfrag{Y}{\small Yes}
\psfrag{N}{\small No}
\psfrag{C}{\begin{boxit2.5} \tiny Theorem \ref{t:CB^r(H)}: $\exists$ (real) $\bdy$-compressions $\DD$ s.t.\\ $\ind\plex{\CC(H/\DD)} \le \ind\plex{\CC(H)}-|\DD|+1$.\end{boxit2.5}}
\psfrag{D}{\small Replace $H$ with $H/\DD$.}
\psfrag{E}{\begin{boxit2.5} \tiny Theorem \ref{t:GammaBB(H)disconnectedToGammaBB(H1)disconnected}: Can isotope $H$ so that \[\ind\plex{\CC \BB^r(H,\TT^1)} \le \ind\plex{\CC \BB^r(H)}.\]\end{boxit2.5}}
\psfrag{H}{\tiny Is $\ind\plex{\CC \EE(H,\TT^1)} \le \ind\plex{\CC \BB^r(H,\TT^1)}$?}
\psfrag{I}{\begin{boxit2.5} \tiny Theorem \ref{t:CE^r(H,T^1)}: $\exists$ interior edge-compressions $\DD$ s.t.\\ $\ind\plex{\CC\BB^r(H\ / \hspace{-.09 in} /\ \DD, \TT^1)} \le \ind\plex{\CC \BB^r(H,\TT^1)}-|\DD|+1$.\end{boxit2.5}}
\psfrag{F}{\small Replace $H$ with $H \slash \DD$.}
\psfrag{J}{\begin{boxit3} \tiny Theorem \ref{t:2skeleton} and Lemma \ref{c:2skeleton}: Can isotope $H$ so that \[\ind\plex{\CC \EE(H,\TT^2)} \le \ind\plex{\CC \EE(H,\TT^1)}.\]\end{boxit3}}
\includegraphics[width=4.5 in]{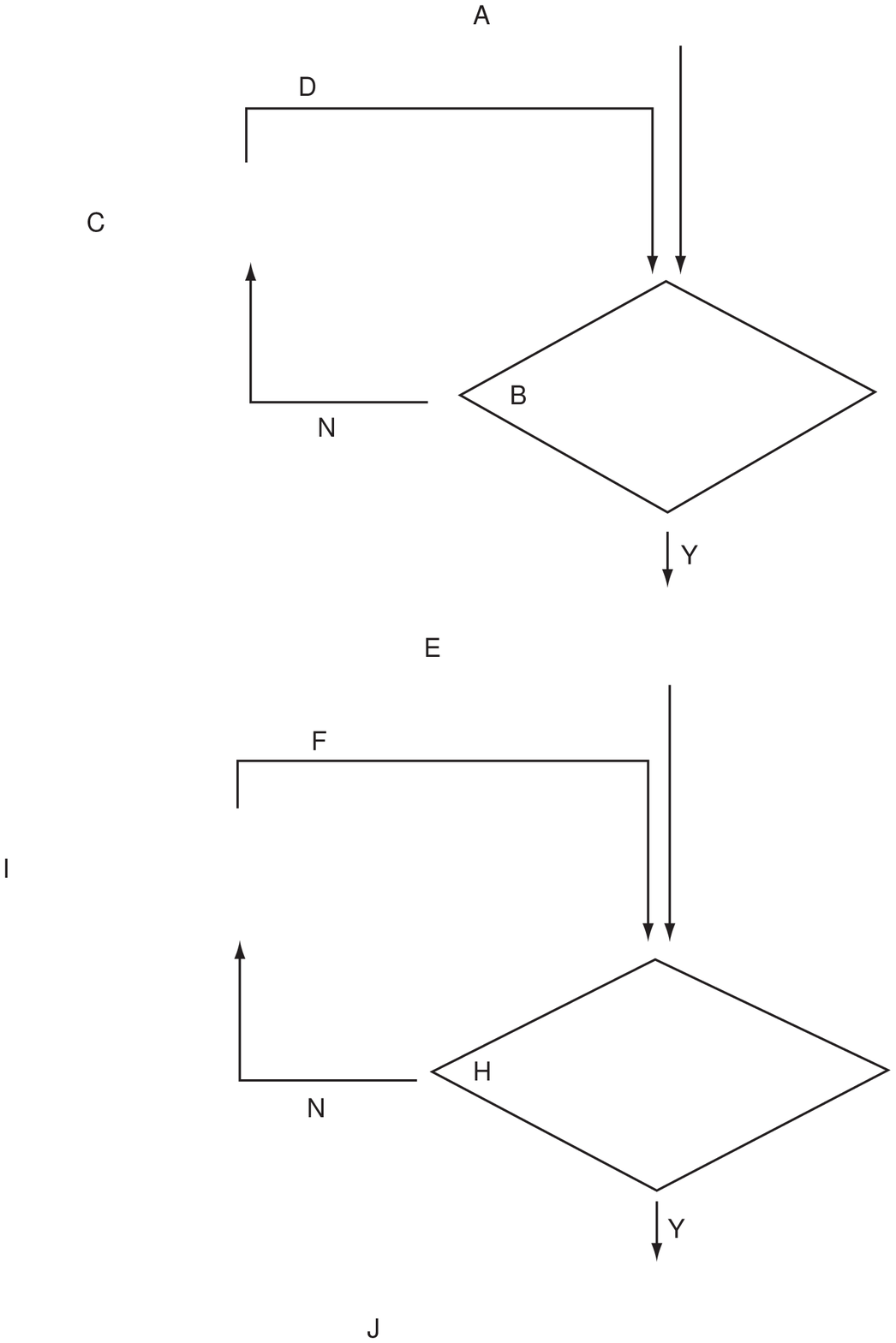}
\caption{Transitioning from $\plex{\CC(H_0)}$ to $\plex{\CC \EE(H,\TT^2)}$ in Theorem \ref{t:main1}.}
\label{f:flowchart}
\end{center}
\end{figure}

The proof of Theorem \ref{t:main_paraphrase} is then complete by applying the following theorem.

\bigskip

\noindent {\bf Theorem \ref{t:IndexSum}.} {\it Suppose $\ind\plex{\CC \EE(H,\TT^2)} = n$. Then 
\[\sum \limits _{\Delta \in \TT^3} \ind\plex{\CC \EE(H \cap \Delta,\TT^2)} = n.\]}

\bigskip

Some remarks are in order about the flow chart of Figure \ref{f:flowchart}. 

\begin{rmk}
The first loop of the flow chart concerns the set of $\bdy$-compressions, and thus only applies to surfaces with non-empty boundary. Theorem~\ref{t:CB^r(H)} in this loop always assumes the surface $H$ that it is given has essential boundary on $M$, and produces a surface $H/\DD$ that has essential boundary. When the genus of $\bdy M \ge 2$, this allows us to bound the distance (as measured in the curve complex) between the boundary curves of the initial surface $H_0$ and the boundary curves of the surface that is eventually passed on to Theorem~\ref{t:GammaBB(H)disconnectedToGammaBB(H1)disconnected} in terms of the Euler characteristic of $H_0$. 
\end{rmk}

\begin{rmk}
When $\bdy M \cong \mathbb T^2$ the size of the collection $\DD$ given by Theorem~\ref{t:CB^r(H)} must be even. This is because each loop of both $\partial H$ and $\bdy (H/\DD)$ must be essential on $\bdy M$, and thus the Euler characteristics of the components of $\bdy M -H$ and $\bdy M-(H/\DD)$ are all zero. Since surgery along a $\bdy$-compressing disk on one side of $H$ decreases the Euler characteristic of the components of $\bdy M-H$ on that side by precisely one and increases the Euler characteristic on the other side by one, there must be the same number of disks in $\DD$ on each side of $H$. Hence, $|\DD|$ is even. 
	
In this case one would measure the distance between the curves $\bdy H$ and $\bdy (H/\DD)$ in the Farey graph. Since $|\DD|$ is even it must be at least 2, and thus it follows from Theorem~\ref{t:CB^r(H)} that the distance between the boundary curves of the initial surface $H_0$ and the boundary curves of the surface that is eventually passed on to Theorem~\ref{t:GammaBB(H)disconnectedToGammaBB(H1)disconnected} is bounded by $\ind\plex{\CC(H_0)}$. In particular, when $\ind\plex{\CC(H_0)}=2$, then these two sets of curves are at most distance 2 apart in the Farey graph. We make extensive use of this fact in \cite{HeegaardDehn}. 
\end{rmk}

\begin{rmk}
In the second loop of the flow-chart, the collection $\DD$ produced by Theorem~\ref{t:CE^r(H,T^1)} has the property that $H/\DD$ meets the neighborhood of each interior edge of $\TT^1$ in an essential loop, and thus $H \slash \DD$ is well-defined. As in the previous remark, this implies that $|\DD|$ is even. When $|\DD|=2$ the surface $H \slash \DD$ is by construction isotopic to $H$ in $M$. In particular, if the original surface $H_0$ is closed, and $\ind\plex{\CC(H_0)} \le 2$,  then the final surface $H$ produced by the flow chart will be isotopic to $H_0$. This is why Theorem \ref{t:classical} is a special case of Theorem \ref{t:main_paraphrase}.
\end{rmk}

\begin{rmk}
The proofs of all but one of the results indicated in the flow chart have appeared with different notation in previous work. In each case below we both indicate the original source, and reproduce the proof, adapted to the notation here. The one exception is Theorem \ref{t:GammaBB(H)disconnectedToGammaBB(H1)disconnected}, which is completely new. We consider this result to be the main technical heart of this paper. It plays a particularly important role here, as it links the two loops of the flow chart.
\end{rmk}

\section{Stage 1: $\plex{\CC(H)} \rightarrow \plex{\CC \BB^r(H)}$}

\begin{lem}
\label{l:generalization} {\rm [cf. \cite{Stabilizing}, Theorem 4.5.]}
Let $X$ be a subcomplex of a simplicial complex $Y$. Suppose the homotopy index of $X$ is $n$. Then either the homotopy index of $Y$ is at most $n$, or there is a simplex $\tau \subset  Y \setminus X$ such that the homotopy index of the subcomplex spanned by
\[\{x \in X|\forall y \in \tau, \ x\mbox{ is adjacent to }y\}\] 
is at most $n-{\rm dim}(\tau)$.
\end{lem}

\begin{proof}
For every simplex $\tau$ of $Y$,  let 
\[V_{\tau}=\{x \in X|\forall y \in (\tau \setminus X), \ x\mbox{ is adjacent to }y\}.\] 
Note that if $\tau \subset X$ then $\tau \setminus X=\emptyset$, and thus $V_{\tau}$ will be all of $X$. 

For all $\tau \subset Y$, let $\tau^+= \tau \setminus X$. Our goal is to show that there is a simplex $\tau \subset Y$ with $\tau^+ \ne \emptyset$ such that the homotopy index of $V_\tau=V_{\tau^+}$ is at most $n-{\rm dim}(\tau^+)$. If $X=Y$ then the result is immediate. If not, then there exists a simplex $\tau$ such that $\tau^+ \ne \emptyset$. By way of contradiction, we suppose that for all $\tau \subset Y$, $V_\tau$ does not have homotopy index less than or equal to $n-{\rm dim}(\tau^+)$. Thus, for all $\tau$, $V_\tau \ne \emptyset$ and 
\begin{equation}
\label{e:contradictionRevisited}
\pi_i(V_\tau) =1 \mbox{ for all } i \le  n-{\rm dim}(\tau^+)-1. 
\end{equation}

\begin{clm}
\label{c:subsetRevisited}
Suppose $\tau$ is a cell of $Y$ which lies on the boundary of a cell $\sigma$. Then $V_{\sigma} \subset V_{\tau}$. 
\end{clm}

\begin{proof}
Suppose $v \in V_{\sigma}$. Then $v \in X$ and $v$ is connected by an edge to every vertex of $\sigma \setminus X$. Since $\tau$ lies on the boundary of $\sigma$, it follows that $\tau \setminus X \subset \sigma \setminus X$, and thus $v$ is connected by an edge to every vertex of $\tau \setminus X$. Thus, $v \in V_\tau$. 
\end{proof}

If $X = \emptyset$ then the result is immediate, as $V_{\tau}$ will be $\emptyset$ (and thus have homotopy index 0) for any simplex $\tau$ in $Y$. 

If $X \ne \emptyset$ then, by assumption, there is a non-trivial map $\iota$ from an $(n-1)$-sphere $S$ into the $(n-1)$-skeleton of  $X$. Assuming the theorem is false will allow us to inductively construct a map  $\Psi$ of a $n$-ball $B$ into $X$ such that $\Psi(\bdy B)=\iota(S)$. The existence of such a map contradicts the non-trivialty of $\iota$. 

Since $X \subset Y$,  $\iota$ is also a map from $S$ into $Y$.  If  $\pi_{n-1}(Y) \ne 1$ then the result is immediate. Otherwise,  $\iota$ can be extended to a map from an $n$-ball $B$ into $Y$. Let $\Sigma$ denote a triangulation of $B$ so that the map $\iota :B \to Y$ is simplicial. 

Push the triangulation $\Sigma$ into the interior of $B$, so that $\rm{Nbhd}(\bdy B)$ is no longer triangulated (Figure \ref{f:SigmaDual}(b)). Then extend $\Sigma$ to a cell decomposition over all of $B$ by forming the product of each cell of $\Sigma \cap S$ with the interval $I$ (Figure \ref{f:SigmaDual}(c)). Denote this cell decomposition as $\Sigma'$. Note that $\iota$ extends naturally over $\Sigma'$, and the conclusion of Claim \ref{c:subsetRevisited} holds for cells of $\Sigma'$. Now let $\Sigma^*$ denote the dual cell decomposition of $\Sigma'$ (Figure \ref{f:SigmaDual}(d)). This is done in the usual way, so that there is a correspondence between the $d$-cells of $\Sigma ^*$ and the $(n-d)$-cells of $\Sigma'$. Note that, as in the figure, $\Sigma ^*$ is not a cell decomposition of all of $B$, but rather a slightly smaller $n$-ball, which we call $B'$. 

\begin{figure}
\psfrag{a}{(a)}
\psfrag{b}{(b)}
\psfrag{c}{(c)}
\psfrag{d}{(d)}
\begin{center}
\includegraphics[width=5 in]{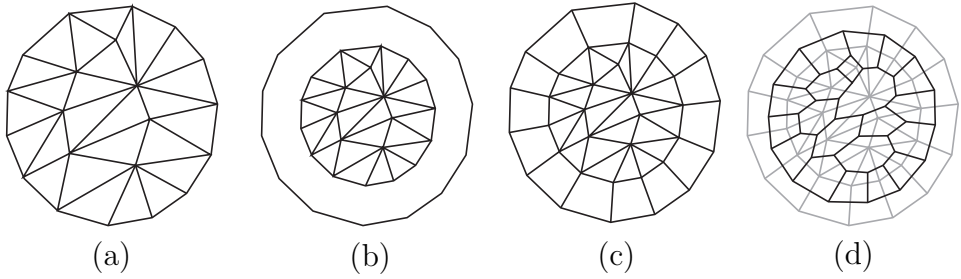}
\caption{(a) The triangulation $\Sigma$ of $B$. (b) Push $\Sigma$ into the interior of $B$. (c) Fill in $\rm{Nbhd}(\bdy B)$ with product cells to complete $\Sigma'$. (d) $\Sigma^*$ is the dual of $\Sigma'$.}
\label{f:SigmaDual}
\end{center}
\end{figure}

For each cell $\tau$ of $\Sigma'$, let $\tau^*$ denote its dual in $\Sigma^*$, and let $V_{\tau}=V_{\iota(\tau)}$. Thus, it follows from Claim \ref{c:subsetRevisited} that if $\sigma^*$ is a cell of $\Sigma^*$ that is on the boundary of $\tau^*$, then $V_{\sigma} \subset V_\tau$. 

We now produce a contradiction by defining a continuous map $\Psi:B' \to X$ such that $\Psi(\bdy B')=\iota(S)$. The map will be defined inductively on the $d$-skeleton of $\Sigma^*$ so that the image of every cell $\tau^*$ is contained in $V_\tau$. 

For each $0$-cell $\tau^* \in \Sigma^*$, choose a point in $V_\tau$ to be $\Psi(\tau^*)$. If $\tau^*$ is in the interior of $B'$ then this point may be chosen arbitrarily in $V_\tau$. If $\tau^* \in \bdy B'$ then $\tau$ is an $n$-cell of $\Sigma'$. This $n$-cell is $\sigma \times I$, for some $(n-1)$-cell $\sigma$ of $\Sigma \cap S$. But since $\iota(S) \subset X$, it follows that $\iota(\tau) \subset X$, and thus $V_\tau=X$. We conclude $\iota(\tau) \subset V_\tau$, and thus we can choose $\iota(\tau^*)$, the barycenter of $\iota(\tau)$, to be $\Psi(\tau^*)$. 

We now proceed to define the rest of the map $\Psi$ by induction. Let $\tau^*$ be a $d$-dimensional cell of $\Sigma^*$ and assume $\Psi$ has been defined on the $(d-1)$-skeleton of $\Sigma^*$. In particular, $\Psi$ has been defined on $\bdy \tau^*$. Suppose $\sigma^*$ is a cell on $\bdy \tau^*$.  By Claim \ref{c:subsetRevisited} $V_\sigma \subset V_\tau$. By assumption $\Psi|\sigma^*$ is defined and $\Psi(\sigma^*) \subset V_\sigma$. We conclude $\Psi(\sigma^*) \subset V_\tau$ for all $\sigma^* \subset \bdy \tau^*$, and thus
\begin{equation}
\label{e:boundary}
\Psi(\bdy \tau^*) \subset V_\tau.
\end{equation}

Note that  \[{\rm dim}(\tau) = n-{\rm dim}(\tau^*)=n-d.\] Since ${\rm dim}(\tau^+) \le {\rm dim}(\tau)$, we have \[{\rm dim}(\tau^+) \le n-d.\] Thus \[d \le n-{\rm dim}(\tau^+),\] and finally \[d-1 \le n-{\rm dim}(\tau^+)-1.\]
It now follows from Equation \ref{e:contradictionRevisited} that $\pi_{(d-1)}(V_\tau)=1$. Since $d-1$ is the dimension of $\bdy \tau^*$, we can thus extend $\Psi$ to a map from $\tau^*$ into $V_\tau$. 

Finally, we claim that if $\tau^* \subset \bdy B'$ then this extension of $\Psi$ over $\tau^*$ can be done in such a way so that $\Psi(\tau^*) =\iota(\tau^*)$. This is because in this case each vertex of $\iota(\tau)$ is in $X$, and hence $V_\tau=X$. As $\iota(S) \subset X=V_\tau$, we have $\iota(\tau^*) \subset V_\tau$. Thus we may choose $\Psi(\tau^*)$ to be $\iota(\tau^*)$.
\end{proof}

\begin{thm}
\label{t:CB^r(H)}
Let $H$ be a non-peripheral, topologically minimal surface in $M$. Then either $\ind\plex{\CC \BB^r(H)} \le \ind\plex{\CC(H)}$, or there exists a collection $\DD$ of pairwise disjoint, real $\bdy$-compressing disks for $H$ such that 
\[\ind\plex{\CC(H/\DD)} \le \ind\plex{\CC(H)} -|\DD|+1.\]
\end{thm}

\begin{proof}
Set $X=\plex{\CC(H)}$, $Y=\plex{\CC \BB^r(H)}$, and apply Lemma \ref{l:generalization}. The conclusion of the lemma gives us two possibilities. The first is that $\ind\plex{\CC \BB^r(H)} \le \ind\plex{\CC(H)}$. The second is that there is a simplex $\tau \subset \plex{\BB^r(H)}$ such that the homotopy index of the complex $Z$ spanned by 
\[\{x \in \plex{\CC(H)}|\forall y \in \tau, \ x\mbox{ is adjacent to }y\}\] 
is at most $\ind \plex{\CC(H_m)} - {\rm dim}(\tau)$. Let $\DD$ denote a pairwise disjoint collection of disks representing the vertices of $\tau$. Then ${\rm dim}(\tau)=|\DD|-1$. We claim that $Z$ is precisely $\plex{\CC(H/\DD)}$, and thus we have obtained the desired conclusion. 

Note that the vertices of $Z$ correspond to the set of compressing disks for $H$ that are disjoint from every disk in $\DD$. Let $H'$ denote the surface obtained from $H$ by simultaneous surgery along the disks in $\DD$, so that $H/\DD$ is obtained from $H'$ by removing any component that lies in a ball. (If $H/\DD=\emptyset$ then the surface $H$ would have been peripheral, a contradiction.) It follows immediately that $\CC(H')=Z$. What remains is to show $\CC(H')=\CC(H/\DD)$.

We now claim that every boundary component of every non-disk component of $H'$ is essential on $\bdy M$. If not, then such a boundary component that is innermost on $\bdy M$ bounds a subdisk $C$ of $\bdy M$ that is isotopic to a compressing disk for $H'$. But, as this disk is isotopic into $\bdy M$, it can be made disjoint from every other disk in $\CC(H')$. Thus, $\plex{\CC(H')}$ would be a contractible complex, a contradiction. 

Suppose now $E$ is a compressing disk for $H'$. Then $E$ meets a component $H''$ of $H'$ that is not a disk. The surface $H''$ will either be closed, and hence had been unaffected by the surgery along  the disks in $\DD$, or by the above argument will have boundary that is essential on $\bdy M$. As $\bdy M$ is assumed to be incompressible, we conclude $\bdy H''$ consists of essential loops in $M$. Hence, $H''$ can not be contained in a ball, and is thus a component of $H/\DD$. We conclude $E$ is a compressing disk for $H/\DD$, and thus $\CC(H')=\CC(H/\DD)$. 
\end{proof}

\section{Stage 2: $\plex{\CC \BB^r(H)} \rightarrow \plex{\CC \BB^r (H,\TT^1)}$}

Theorem \ref{t:CB^r(H)} has left us with a surface $H$ such that $\plex{\CC \BB^r(H)}$ has well-defined homotopy index. The goal of this section is to show that $H$ can be isotoped so that  $\ind \plex{\CC \BB^r (H,\TT^1)} \le \ind \plex{\CC \BB^r(H)}$

Henceforth, we will assume that $H$ has been isotoped so that $|\bdy H \cap \TT^1|$ is minimal. This has the following nice consequence. 

\begin{lem}
\label{l:BoundaryDisksAreHonest}
If $|\bdy H \cap \TT^1|$ is minimal and $B \in \BB^r(H,\TT^1)$, then $B \in \BB^r(H)$. 
\end{lem}

\begin{proof}
If $B \notin \BB^r(H)$ then either $B$ is not a $\bdy$-compressing disk for $H$, or $B$ is fake. In the former case $B$ cuts off a subdisk $C$ of $H$. As $\bdy M$ is incompressible, and $C \cup B$ is a disk incident to $\bdy M$, it follows that $B$ meets $\bdy M$ in an arc $\beta$ that is parallel to an arc $\alpha$ of $\bdy H$. By definition, this is also true when $B$ is fake. Since $B \in \BB^r(H,\TT^1)$, $\beta$ does not meet $\TT^1$. Hence, isotoping $\alpha$ to $\beta$ reduces $|\bdy H \cap \TT^1|$, a contradiction. 
\end{proof}

\begin{dfn}
Suppose $D \in \CC \BB^r(H,\TT^1)$. If $D \in \CC \BB^r(H)$ then we say it is {\it honest}. Otherwise, we say $D$ is {\it dishonest}. 
\end{dfn}

By Lemma \ref{l:BoundaryDisksAreHonest}, if $|\bdy H \cap \TT^1|$ is minimal and $D \in \BB^r(H,\TT^1)$ then $D$ must be honest. It follows that if $D$ is a dishonest element of $\CC \BB^r (H,\TT^1)$ then $D \in \CC(H,\TT^1)$ but $D \notin \CC(H)$. Thus, $\bdy D$ bounds a subdisk $C$ of $H$ whose interior meets $\TT^1$.

\begin{lem}
\label{l:HonestAfterDishonest}
Suppose $D$ and $E$ are disjoint elements of $\CC \BB^r(H,\TT^1)$ and $D$ is dishonest. Then $E \in \CC \BB^r(H/D,\TT^1)$.
\end{lem}

\begin{proof}
$H/D$ is obtained from $H$ by removing a neighborhood of $D \cap H$, and replacing it with two copies $D'$ and $D''$ of $D$. If the lemma is false, then $E \cap (H/D)$ cuts off a subdisk $E^*$ of $H/D$ that is disjoint from $\TT^1$. If $E^*$ does not contain $D'$ or $D''$, then $E$ would not have been a compression for $H-N(\TT^1)$. If $E^*$ contains both $D'$ and $D''$, then $D$ would have been honest. Finally, if $E^*$ contains one of $D'$ or $D''$ (assume the former), then $D \cap H$ is isotopic to $E \cap H$ on $H-N(\TT^1)$. When we isotope these intersections to coincide, then $D \cup E$ becomes either a sphere in $M-N(\TT^1)$, or a disk with boundary on $\bdy M-N(\TT^1)$. In either case, $D \cup E$ cuts a ball $B$ out of $M$. Since $D$ is assumed to be dishonest, it follows that $E$ is dishonest as well. Thus, the ball $B$ must contain all of $\TT^1$, a contradiction. 
\end{proof}

We now define a complexity on $H$, and show that both honest and dishonest compressions and $\bdy$-compressions decrease complexity. 

\begin{dfn}
If $H$ is empty, then we define the {\it width} $w(H)$ to be $(0,0)$. If $H$ is connected, then the {\it width} $w(H)$ is the pair $(-\chi(H), |\TT^1 \cap H|)$. If $H$ is disconnected, then its {\it width} is the ordered set of the widths of its components, where we include repeated pairs and the ordering is non-increasing. Comparisons are made lexicographically at all levels.
\end{dfn}

\begin{lem}
\label{l:widthchange}
Suppose $D\in \CC \BB^r(H,\TT^1)$. Then $w(H/D)<w(H)$.
\end{lem}

\begin{proof}
Suppose first that $D$ is honest, so that $D \cap H$ is essential in $H$. If, furthermore, $\bdy D$ does not separate $H$, then $-\chi(H/D)$ is less than $-\chi(H)$, and hence the width is also less. If, on the other hand, $\bdy D$ separates $H$ then $H/D$ is disconnected, and both components have smaller negative Euler characteristic. Hence, again width has decreased.

Now suppose $D$ is dishonest. Then $\bdy D$ cuts off a subdisk $D'$ of $H$ whose interior meets $\TT^1$.  Thus, $D' \cap \TT^1 \ne \emptyset$. The result of surgering along $D$ produces a sphere and a surface that has the same Euler characteristic as $H$, but meets $\TT^1$ fewer times, and thus has smaller width than $H$. This latter surface is precisely $H/D$.
\end{proof}

\begin{dfn}
For any simplicial complex $X$, let $X^{(1)}$ denote the first barycentric subdivision of $X$. If $v$ is a vertex of $X^{(1)}$, then $v^*$ will denote the simplex of $X$ that $v$ is the barycenter of. 
\end{dfn}

\begin{dfn}
\label{d:CompressingSequence}
A {\it compressing $n$-sequence} $\Sigma$ for a surface $H$ is a simplicial $n$-complex (homeomorphic to $D^2$ when $n=2$) where each vertex $v$ is associated with a surface $H_v$ in $M$ which is either isotopic to $H$ in $M$, or is obtained from $H$ by compression or $\bdy$-compression. Let $\Sigma_0$ denote the subcomplex of $\Sigma$ spanned by those vertices $v$ for which $H_v$ is isotopic to $H$ in $M$. $\Sigma$ admits compatible global and local structures, as follows.

\medskip

\begin{figure}
\begin{center}
\psfrag{D}{$D_1$}
\psfrag{E}{$D_2$}
\psfrag{F}{$D_3$}
\psfrag{a}{$a$}
\psfrag{b}{$b$}
\psfrag{c}{$c$}
\psfrag{A}{$\iota(a)$}
\psfrag{B}{$\iota(b)$}
\psfrag{C}{$\iota(c)$}
\psfrag{H}{$H_a$}
\psfrag{I}{$H_b$}
\psfrag{J}{$H_c$}
\psfrag{i}{$\iota$}
\psfrag{K}{$H$}
\psfrag{S}{$\Sigma$}
\psfrag{G}{$\plex{\CC \BB^r(H)}^{(1)}$}
\includegraphics[width=5in]{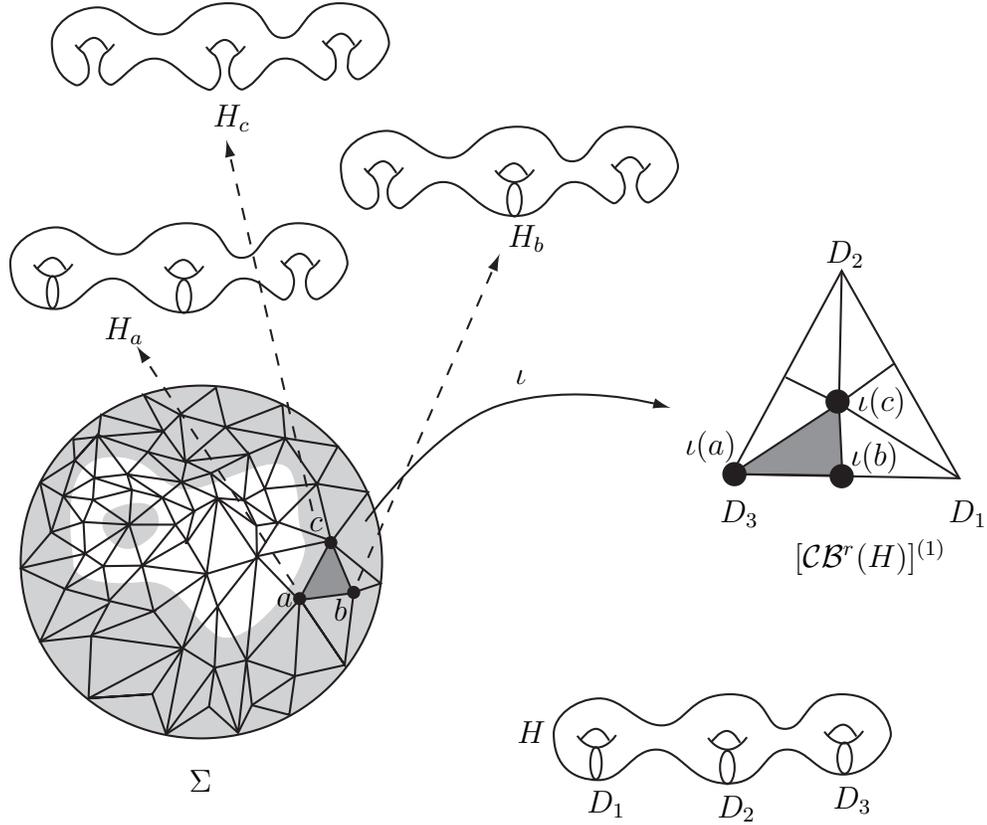}
\caption{The  unshaded region of $\Sigma$ represents $\Sigma_0$. The lighter shaded region is $\Sigma -\Sigma_0$, the domain of $\iota$.}
\label{f:Global}
\end{center}
\end{figure}

\noindent \underline{Global Structure}

	\begin{enumerate}
		\item When $v \in \Sigma_0$ there is a fixed isotopy from $H_v$ to $H$, inducing an identification of $\plex{\CC \BB^r(H_v)}$ with $\plex{\CC \BB^r(H)}$.  
		\item There is a simplicial map $\iota:\Sigma -\Sigma_0 \to \plex{\CC \BB^r(H)}^{(1)}$, such that for each $v$ in the domain of $\iota$ the surface $H_v$ is isotopic in $M$ to  $H/ \iota(v)^*$. See Figure \ref{f:Global}.
	\end{enumerate}

\medskip

\begin{figure}
\begin{center}
\psfrag{D}{$D_1$}
\psfrag{E}{$D_2$}
\psfrag{F}{$D_3$}
\psfrag{a}{$\iota(w)$}
\psfrag{c}{$\eta_v(w)$}
\psfrag{w}{$w$}
\psfrag{v}{$v$}
\psfrag{n}{$\eta_v$}
\psfrag{i}{$\iota$}
\psfrag{K}{$H_v$}
\psfrag{H}{$H_w$}
\psfrag{T}{$\TT^1$}
\psfrag{U}{$\mbox{link}(v)^-$}
\psfrag{S}{$\Sigma$}
\psfrag{G}{$\plex{\CC \BB^r(H_v,\TT^1)}^{(1)}$}
\includegraphics[width=3.5in]{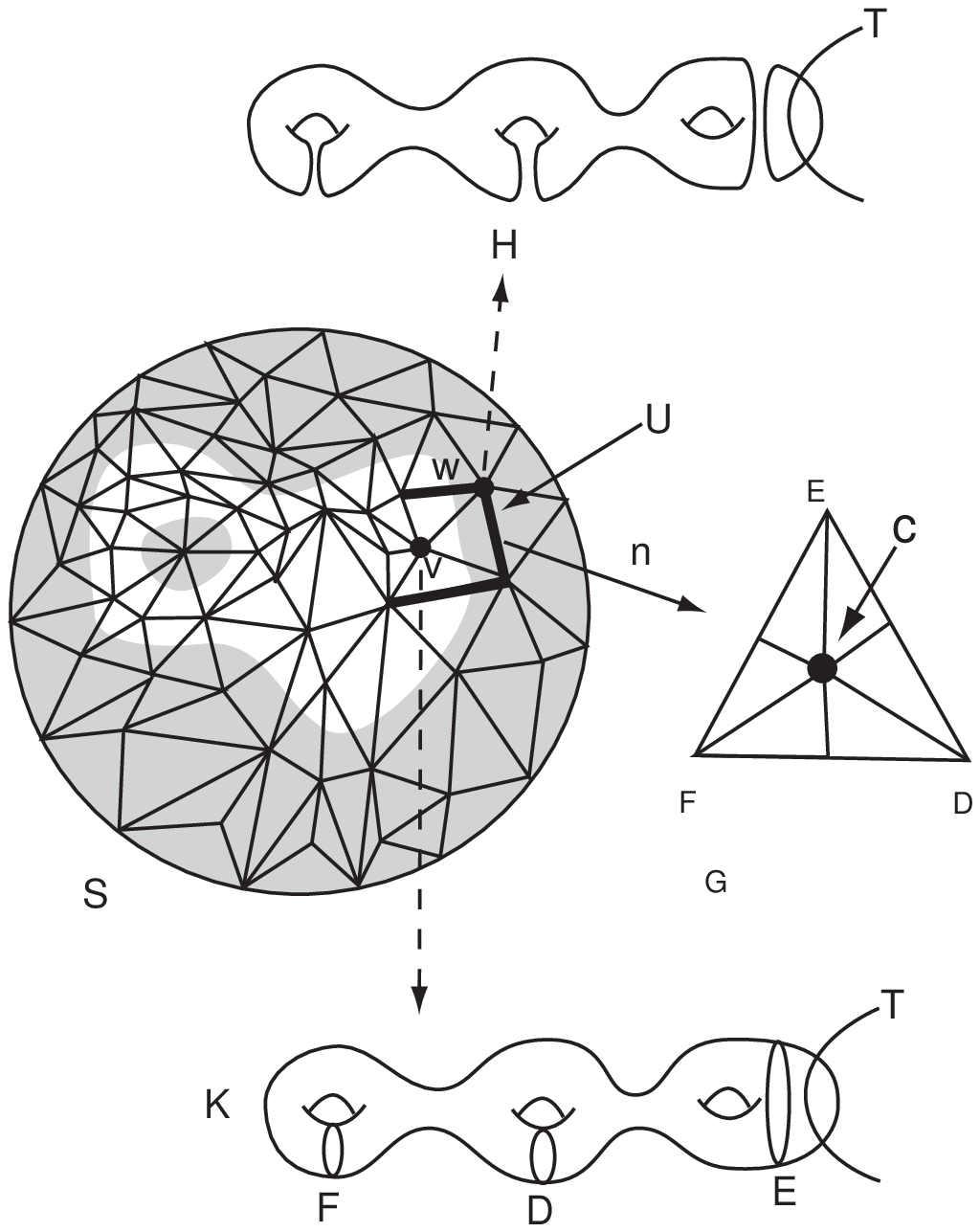}
\caption{$\mbox{link}(v)^-$ is the subcomplex of $\mbox{link}(v)$ spanned by vertices $u$ for which $w(H_u)<w(H_v)$. The map $\eta_v$ sends $\mbox{link}(v)^-$ into $\plex{\CC \BB^r(H_v,\TT^1)}^{(1)}$.}
\label{f:Local}
\end{center}
\end{figure}

\noindent \underline{Local Structure}

	\begin{enumerate}
		\item If $e$ is an edge of $\Sigma$ with endpoints $u$ and $v$, where $H_u, H_v \in \Sigma_0$ and $w(H_u)=w(H_v)$, then $H_u=H_v$.
		\item Let $v$ be a vertex in $\Sigma_0$, and $\mbox{link}(v)^-$ the subcomplex of $\mbox{link}(v)$ spanned by vertices $u$ such that $w(H_u)<w(H_v)$. Then there is a simplicial map $\eta_v:\mbox{link}(v)^- \to \plex{\CC \BB^r(H_v,\TT^1)}^{(1)}$ such that for each $w \in \mbox{link}(v)^-$, $H_w=H_v/\eta_v(w)^*$.\footnote{Equality here indicates that there exists a set of disk representatives for the vertices of $\eta_v(w)^*$ so that $H_w$ is precisely the surface obtained from $H_v$ by surgering along these disks.} See Figure \ref{f:Local}.
	\end{enumerate}

\medskip

\begin{figure}
\begin{center}
\psfrag{D}{$D_1$}
\psfrag{E}{$D_2$}
\psfrag{F}{$D_3$}
\psfrag{a}{$\iota(w)$}
\psfrag{c}{$\eta_v(w)$}
\psfrag{v}{$v$}
\psfrag{w}{$w$}
\psfrag{n}{$\eta_v$}
\psfrag{i}{$\iota$}
\psfrag{K}{$H_v$}
\psfrag{T}{$\TT^1$}
\psfrag{S}{$\Sigma$}
\psfrag{P}{$\plex{\CC \BB^r(H)}^{(1)}$}
\psfrag{G}{$\plex{\CC \BB^r(H_v,\TT^1)}^{(1)}$}
\includegraphics[width=5in]{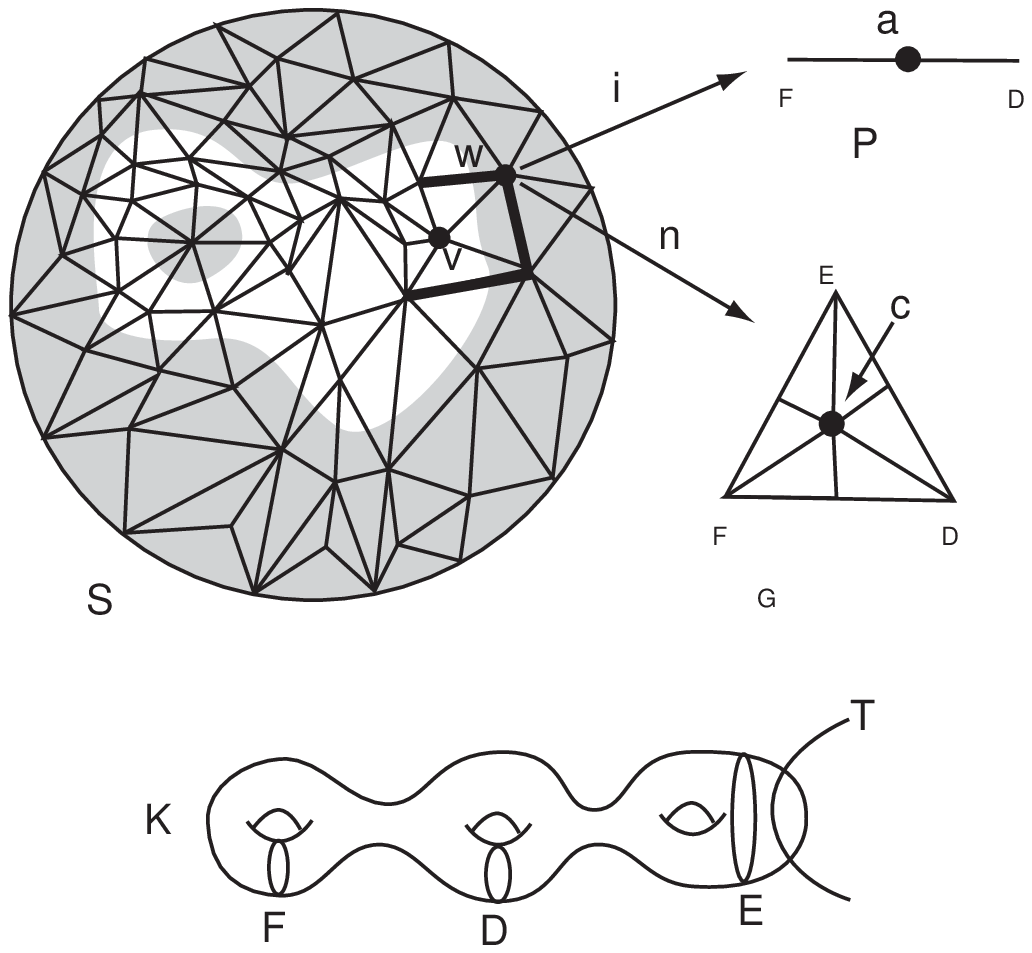}
\caption{The compatibility condition implies that $\iota(w)$ is at the barycenter of the simplex of $\plex{\CC \BB^r(H)}$ spanned by the vertices of $\eta_v(w)^*$ corresponding to honest compressing disks.}
\label{f:Compatibility}
\end{center}
\end{figure}

\noindent \underline{Compatibility Condition}

For each vertex $v$ in $\Sigma_0$,  and each $w \in \mbox{link}(v)^-$, the simplex of $\plex{\CC \BB^r(H)}$ spanned by the vertices of $\eta_v(w)^*$ that represent honest compressing disks for $H$ is precisely the simplex $\iota(w)^*$. See Figure \ref{f:Compatibility}.
\end{dfn}

\bigskip

\begin{dfn}
Suppose $\Sigma$ is a compressing $n$-sequence, $v \in \Sigma_0$, and $V$ is a component of the subcomplex of  $\Sigma$ such that $H_w=H_v$ for all $w \in V$. Let $\rm{star}(V)$ denote the union of the simplices of $\Sigma$ that meet $V$, $\overline{\rm{star}}(V)$ the closure of $\rm{star}(V)$, and  ${\rm link}(V)=\overline{\rm{star}}(V)-\rm{star}(V)$. Then we say $V$ is a {\it plateau} if $w(H_v)>w(H_u)$ for all $u \in {\rm link}(V)$.
\end{dfn}

Now suppose $V$ is a plateau, and $v$ is a vertex of $V$. Then note that the maps $\eta _w$, for each $w \in V$, induce a simplicial map $\eta_V:{\rm link}(V) \to \plex{\CC \BB^r(H_v,\TT^1)}^{(1)}$. See Figure \ref{f:ThickPlateau}.

\begin{figure}
\begin{center}
\psfrag{a}{$\iota(w)$}
\psfrag{c}{$\eta_V(w)$}
\psfrag{w}{$w$}
\psfrag{v}{$v$}
\psfrag{n}{$\eta_V$}
\psfrag{i}{$\iota$}
\psfrag{K}{$H_v$}
\psfrag{H}{$H_w$}
\psfrag{T}{$\TT^1$}
\psfrag{U}{$\mbox{link}(v)^-$}
\psfrag{S}{$\Sigma$}
\psfrag{G}{$\plex{\CC \BB^r(H_v,\TT^1)}^{(1)}$}
\includegraphics[width=3.5in]{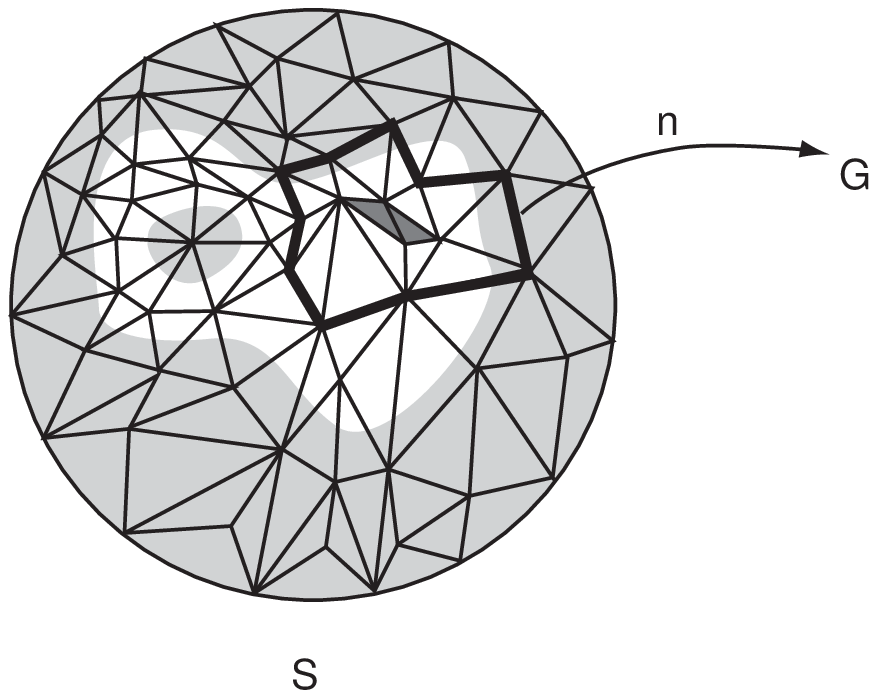}
\caption{The darker shaded region represents a plateau $V$. The map $\eta_V$ sends $\mbox{link}(V)$ into $\plex{\CC \BB^r(H_v,\TT^1)}^{(1)}$, where $v$ is any vertex of $V$.}
\label{f:ThickPlateau}
\end{center}
\end{figure}

\begin{dfn}
Let $S$ be an $(n-1)$-cycle in the singular homology of $\plex{\CC \BB^r(H)}^{(1)}$. We say a compressing $n$-sequence $\Sigma$ for $H$ {\it spans $S$} if $\iota(\bdy \Sigma)=S$.
\end{dfn}

\begin{lem}
\label{l:ThickExists}
Suppose $\Sigma$ is a compressing $n$-sequence that spans an $(n-1)$-cycle $S$. If $n=2$ and $S$ is non-trivial in $\pi_1(\plex{\CC \BB^r(H)}^{(1)})$, or $n \ne 2$ and $S$ is non-trivial in $H_{n-1}(\plex{\CC \BB^r(H)}^{(1)})$, then $\Sigma$ has a plateau $V$. 
\end{lem}

\begin{proof}
First, note that $\Sigma_0$ can not be empty because otherwise $\iota$ would map all of $\Sigma$ into $\plex{\CC \BB^r(H)}^{(1)}$. When $n=2$ this gives a homotopy of $S$ to a point, and when $n \ne 2$ this gives a singular $n$-chain that $S$ is the boundary of. Now let $v$ denote a vertex of $\Sigma_0$, for which $w(H_v)$ is maximal. Let $V$ denote the component of the subcomplex of $\Sigma$ spanned by vertices $w$ with $H_w=H_v$, containing $v$. Then $V$ must be a plateau. 
\end{proof}

\begin{dfn}
Let $\Sigma$ be a compressing $n$-sequence. Let $\{v_i\}$ denote a choice of vertex in each plateau. Then the {\it size} of $\Sigma$ is the ordered set $\{w(H_{v_i})\}$, where the ordering is non-increasing and comparisons are made lexicographically. 
\end{dfn}

By Lemma \ref{l:ThickExists}, the size of a compressing $n$-sequence is well-defined, as long as it spans a homotopically/homologically non-trivial cycle $S$.

\begin{lem}
\label{l:ThickGivesResult}
Let $S$ be an $(n-1)$-cycle, and $\Sigma$ a minimal size compressing $n$-sequence that spans $S$. Let $V$ be a  plateau of $\Sigma$. If $n=2$ and $S$ is non-trivial in $\pi_1(\plex{\CC \BB^r(H)}^{(1)})$ then $\eta_V({\rm link}(V))$ is non-trivial in $\pi_1(\plex{\CC \BB^r(H_v,\TT^1)}^{(1)})$. If $n \ne 2$ and $S$ is non-trivial in $H_{n-1}(\plex{\CC \BB^r(H)}^{(1)})$, then $\eta_V({\rm link}(V))$ is non-trivial in $H_{n-1}(\plex{\CC \BB^r(H_v,\TT^1)}^{(1)})$.
\end{lem}

\begin{proof}
We first address the case where $n=2$. Let $\alpha$ be a component of ${\rm link}(V)$ that is outermost on the disk $\Sigma$. Let $A$ be the subdisk of $\Sigma$ bounded by $\alpha$, and note that $V \subset A$. If the lemma is false, then $\eta_V(\alpha)$ is trivial in $\pi_1(\plex{\CC \BB^r(H_v,\TT^1)}^{(1)})$. Thus, we may extend the map $\eta_V|\alpha$ over a disk $\Pi$. That is, there is a map $\phi:\Pi \to \pi_1(\plex{\CC \BB^r(H_v,\TT^1)}^{(1)})$ such that $\phi(\bdy \Pi)=\eta_V(\alpha)$. In this case, we show we can create a new compressing $n$-sequence of smaller size by replacing the disk $A$ in $\Sigma$ with $\Pi$. 

If $n \ne 2$ and the lemma is false, then $\eta_V$ can be extended to a simplicial map $\phi$ from a simplicial $n$-complex $\Pi$ into $\plex{\CC \BB^r(H_v,\TT^1)}^{(1)}$ (so that $\phi(\bdy \Pi)=\eta_V({\rm link}(V))$). In this case, we show we can create a new compressing $n$-sequence of smaller size by replacing $\overline{\rm{star}}(V)$ in $\Sigma$ with $\Pi$. 

In either case, to realize the new complex thus obtained as a compressing $n$-sequence, we go through each of the conditions in Definition \ref{d:CompressingSequence}. Choose $v \in V$. We associate each vertex $w$ of $\Pi$ with a surface $H_w$ in $M$ by letting $H_w=H_v/\phi(w)^*$. Note that when $H_w$ is isotopic to $H_v$ (and thus $H$), this gives a particular isotopy. 

To satisfy the global condition of Definition \ref{d:CompressingSequence}, we must extend the map $\iota$ over those vertices $w$ of $\Pi$ that are now associated with surfaces not isotopic to $H$. We will call the new map thus created $\iota'$. This map is defined by letting $\iota'(w)$ be the barycenter of the simplex spanned by the vertices of $\phi(w)^*$ representing honest compressing disks. Since $\phi(\bdy \Pi)=\eta_V(\bdy \Pi)$,  the compatibility condition for $\eta_V$ and the original map $\iota$ guarantees that $\iota'|\bdy \Pi=\iota|\bdy \Pi$.

We now discuss the local condition. Let $w$ be a vertex of $\Pi$ for which $H_w$ is isotopic to $H$ in $M$. Let $u$ be a vertex adjacent to $w$. Since $u$ and $w$ are connected by an edge, and the map $\phi$ is simplicial, we may conclude that $\phi(u)$ and $\phi(w)$ are either the same vertex of $\plex{\CC \BB^r(H_v,\TT^1)}^{(1)}$, or they are connected by an edge in $\plex{\CC \BB^r(H_v,\TT^1)}^{(1)}$. In the former case $H_u$ and $H_w$ will be the same surface. In the latter case one of $\phi(u)^*$ or $\phi(w)^*$ is a face of the other. If $\phi(u)^*$ is a face of $\phi(w)^*$ then by Lemma \ref{l:widthchange}  the width of $H_u=H_v/\phi(u)^*$ is larger than the width of $H_w=H_v/\phi(w)^*$. In this case there is nothing further to show. When $\phi(w)^*$ is a face of $\phi(u)^*$ the width of $H_u$ is less than the width of $H_w$, and hence $u \in \mbox{link}(w)^-$. In this case we must now define $\eta_w(u)$. Since $H_w$ is isotopic to $H$ in $M$, it must be the case that the disks represented by the vertices of $\phi(w)^*$ are dishonest. Hence, by Lemma \ref{l:HonestAfterDishonest} $H_u$ is obtained from $H_w$ by surgering along the simplex $\phi(u)^*-\phi(w)^*$. We may thus define $\eta_w(u)$ to be the barycenter of the simplex $\phi(u)^*-\phi(w)^*$, so that $H_u=H_w/\eta_w(u)^*$. 

Finally, we must check the compatibility of the maps $\eta_w$ and $\iota'$ defined above. Since all of the compressing disks represented by vertices of $\phi(w)^*$ are dishonest, the vertices of $\eta_w(u)^*=\phi(u)^*-\phi(w)^*$ that represent honest compressing disks are just the vertices of $\phi(u)^*$ that represent honest compressing disks.  These vertices span the simplex that $\iota'(u)$ is the barycenter of, as required by the compatibility condition. 
\end{proof}


\begin{thm}
\label{t:GammaBB(H)disconnectedToGammaBB(H1)disconnected}
If $\ind \plex{\CC \BB^r(H)}$ is well defined, then $H$ may be isotoped so that
\[\plex{\CC \BB^r(H,\TT^1)} \le \ind \plex{\CC \BB^r(H)}.\]
\end{thm}

\begin{proof}
By assumption there is a homotopically non-trivial map $\iota$ from an $(n-1)$-sphere $S$ into $\plex{\CC \BB^r(H)}$. When $n \ne 2$, it is a consequence of the Hurewicz Theorem that this map is homologically non-trivial. In any case, our goal is to construct a compressing $n$-cycle that spans $S$. It then follows from Lemmas \ref{l:ThickExists} and \ref{l:ThickGivesResult} that a minimal such compressing $n$-sequence has a vertex $v$ such that $H_v$ is isotopic to $H$ in $M$ and the homotopy index of $\plex{\CC \BB^r(H_v,\TT^1)}$ is at most $n$. The result thus follows. 

Let $\Phi'$ denote a triangulation of $S$ so that the map $\iota$ is simplicial. Let $B$ denote the $n$-ball obtained by coning $S$ to a point $z$. By coning each simplex of $\Phi'$ to $z$ we thus get a triangulation $\Phi$ of $B$. 

We now turn the triangulation $\Phi$ into a cubing $\bar \Phi$ of $B$. Let $\sigma$ be an $(n-1)$-simplex of $\Phi'$ with vertices $\{v_i\}_{i=1}^n$. Let ${\bf e}_i$ denote an orthonormal basis for $\mathbb R^n$, and let $C$ denote the unit cube spanned by these vectors. Then we can identify the simplex $\sigma \ast z$ of $\Phi$ with $C$ by placing 
	\begin{itemize}
		\item $z$ at the origin, 
		\item $v_i$ at the corner of $C$ given by the vector ${\bf e}_i$, and
		\item the barycenter of each face $\sigma'$ of $\sigma$ at the corner of $C$ given by $\sum {\bf e}_i$, where the sum is taken over all $i$ such that $v_i$ is a vertex of $\sigma'$. 
	\end{itemize}
See Figure \ref{f:Simplex2Cube}. Note that the map $\iota$, which sends the simplices of $\Phi'=\Phi|S$ to simplices of $\plex{\CC \BB^r(H)}$, now extends naturally to a map that sends simplices of $\bar \Phi|S$ to simplices of $\plex{\CC \BB^r(H)}^{(1)}$. We continue to denote this map $\iota$. 

\begin{figure}
\psfrag{a}{$v_1$}
\psfrag{b}{$v_2$}
\psfrag{c}{$v_3$}
\psfrag{z}{$z$}
\psfrag{e}{${\bf e}_1$}
\psfrag{f}{${\bf e}_2$}
\psfrag{g}{${\bf e}_3$}
\psfrag{A}{(a)}
\psfrag{B}{(b)}
\psfrag{C}{(c)}
\begin{center}
\includegraphics[width=5 in]{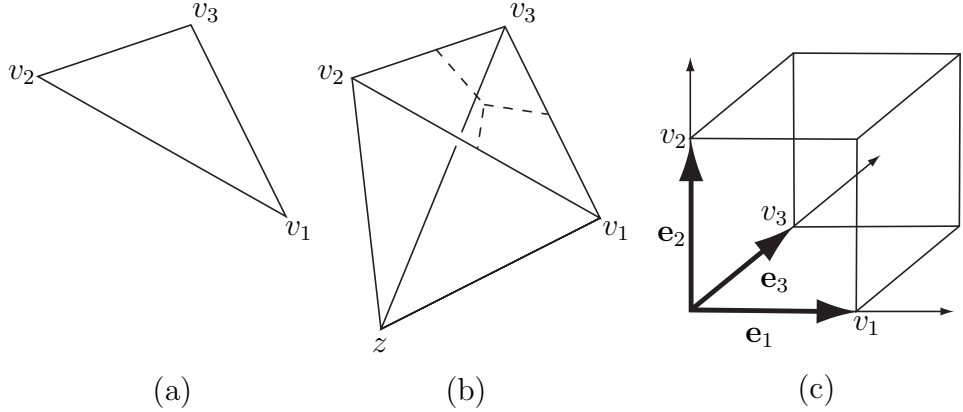}
\caption{(a) A simplex $\sigma$ of $\Phi'$. (b) The simplex $\sigma \ast z$ of $\Phi$. (c) Realizing $\sigma \ast z$ as a cube $C$ of $\bar \Phi$.}
\label{f:Simplex2Cube}
\end{center}
\end{figure}

We now subdivide the cubing $\bar \Phi$ to get a finer cubing of $B$, and associate a surface in $M$ with each vertex of this finer cubing. To begin, choose a representative disk for each vertex of $\iota(S)$, so that if $\tau$ is a simplex  in the image of $\iota$, then the chosen representatives of the vertices of $\tau$ are pairwise disjoint. For the cube $C$ above, let $D_i$ denote the chosen representative of $v_i$. For each $i$, let $n_i=|\TT^1 \cap D_i|$. In each disk $D_i$ choose a collection of arcs $\{\alpha _i^j\}_{j=1}^{n_i}$ such that
	\begin{enumerate}
		\item For each $i$ and $j$, $\alpha_i^j$ connects a point of $D_i \cap \TT^1$ to a point in $\bdy D_i$. 
		\item For each $j$ and $k$ where $j<k$, the arcs $\alpha_i^j$ and $\alpha_i^k$ should either be disjoint, or $\alpha _i^j \subset \alpha _i^k$. 
		\item If $D_i$ meets $\bdy M$, then there is a $j$ such that $\alpha_i^j \subset (D_i \cap \bdy M)$, and contains all of the points of $\bdy D_i \cap \TT^1$.
	\end{enumerate}
Finally, choose a small enough neighborhood $N(D_i)$ so that if $D_i \cap D_j=\emptyset$, then $N(D_i) \cap N(D_j)=\emptyset$, and let $H_i=H \cap N(D_i)$. 

Let $H^0_i=H_i$. For each $1 \le J \le n_i$, let $H_i^J$ denote the surface obtained from $H_i$ by simultaneous surgery along all of the components of $\bigcup \limits _{j=1}^J \alpha _i ^j$. Note that $H_i/D_i$ can then be obtained from $H_i^{n_i}$ by an honest compression or $\bdy$-compression, and for each $j \ge 1$ the surface $H^{j-1}_i$ can be obtained from $H^j_i$ by a dishonest compression or $\bdy$-compression. See Figure \ref{f:SurfaceSequence}.

\begin{figure}
\psfrag{E}{$D_i$}
\psfrag{A}{$H_i=H_i^0$}
\psfrag{B}{$H_i^1$}
\psfrag{C}{$H_i^2$}
\psfrag{D}{$H_i/D_i$}
\begin{center}
\includegraphics[width=5in]{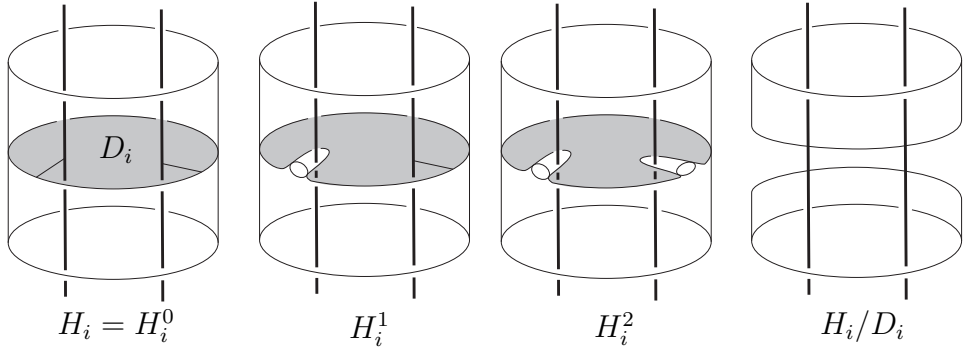}
\caption{The surfaces $H_i^j$ defined by a disk $D_i$.}
\label{f:SurfaceSequence}
\end{center}
\end{figure}

We now subdivide the cube $C$ of $\bar \Phi$ into smaller subcubes. For each $i$, cut $C$ by $n_i$ planes orthogonal to ${\bf e}_i$. Each vertex of this new cubing is then of the form
\[x=\sum \limits _{i=1}^n \frac{x(i)}{n_i+1} {\bf e}_i\]
where $0 \le x(i) \le n_i+1$ for each $i$. We associate to the vertex $x$ the surface $H_x$ in $M$ obtained from $H$ by replacing $H_i$ with 
	\begin{itemize}
		\item $H^{x(i)}_i$, if $x(i) \le n_i$
		\item $H_i/D_i$, if $x(i)=n_i+1$
	\end{itemize}
for each $i$. 

We now triangulate each of the subcubes of $C$. Suppose $x$ and $y$ are vertices of such a subcube $c$ that are connected by an edge. Then $x(i)$ and $y(i)$ differ by exactly one for one value $j$ of $i$, and are equal for every $i \ne j$. Hence, the surfaces $H_x$ and $H_y$ are identical away from $N(D_j)$. Since $x(j)$ and $y(j)$ differ by exactly one, it follows that one of $H^{x(j)}_j$ and $H^{y(j)}_j$ is obtained from the other by a compression in $N(D_j)$. (This compression will be honest if either $x(j)$ or $y(j)$ is $n_j+1$, and dishonest otherwise.) Furthermore, parallel edges of the subcube $c$ will also correspond to pairs of surfaces obtained by compressing along the same disk $D_j$. It follows that for the subcube $c$ there is a vertex $v$, represented by a surface $H_v$, and a collection $\mathcal D$ of honest and dishonest compressing disks for $H_v$, such that every other vertex is obtained from $H_v$ by simultaneous compression along some subcollection of $\mathcal D$. See Figure \ref{f:CubeAndSimplex}(a).

\begin{figure}
\psfrag{A}{$H_v$}
\psfrag{B}{$H_v/A$}
\psfrag{C}{$H_v/AB$}
\psfrag{D}{$H_v/ABC$}
\psfrag{G}{$H_v/B$}
\psfrag{E}{$H_v/C$}
\psfrag{H}{$H_v/AC$}
\psfrag{F}{$H_v/BC$}
\psfrag{a}{(a)}
\psfrag{b}{(b)}
\begin{center}
\includegraphics[width=5in]{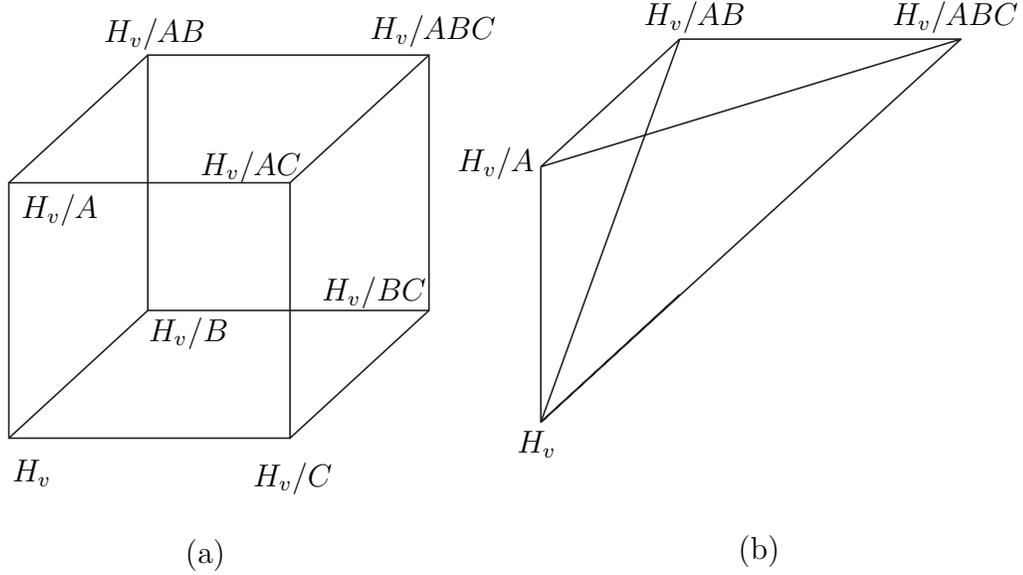}
\caption{(a) The subcube $c$, when $\mathcal D=\{A,B,C\}$. (b) The simplex in $c$ defined by the sequence $\mathcal D_0=\emptyset$, $\mathcal D_1=\{A\}$, $\mathcal D_2=\{A,B\}$, and $\mathcal D_3=\mathcal D$.}
\label{f:CubeAndSimplex}
\end{center}
\end{figure}

We now break up each subcube $c$ of $C$ into a collection of $n$-simplices, thereby obtaining a triangulation of $B$. Let $\{\mathcal D_i\}_{i=0}^n$ denote a sequence of subsets of $\mathcal D$, such that
	\begin{enumerate}
		\item $\mathcal D_n=\mathcal D$
		\item $\mathcal D_i$ is obtained from $\mathcal D_{i+1}$ by removing one disk.
	\end{enumerate}
Thus $\mathcal D_0=\emptyset$ and for each $i \ge 1$,  $\mathcal D_i$ is a collection with precisely $i$ disks. Any such sequence $\{\mathcal D_i\}$ defines a simplex in $c$ spanned by the vertices of $c$ associated to the surfaces $H_v/\mathcal D_i$. See Figure \ref{f:CubeAndSimplex}(b). 

To show that the simplicial complex thus defined, with the given associated surfaces for each vertex, is a compressing $n$-sequence spanning $S$ we make a few observations. First, note that by construction, every vertex besides those in $\bdy B$ are isotopic in $M$ to $H$. Thus, $\Sigma_0$ is everything besides $\bdy B=S$. The original map $\iota$ which sent $S$ into $\plex{\CC \BB^r(H)}$ has been refined by our construction to be a map from $S^{(1)}$ into $\plex{\CC \BB^r(H)}^{(1)}$, and this is precisely the map required by the global condition of Definition \ref{d:CompressingSequence}. 

Now consider a vertex $v$ of $\Sigma_0$. The surface $H_v$ is associated to a surface obtained from $H$ by surgering along some collection of arcs. For each vertex $w$ in $\mbox{link}(v)^-$, the surface $H_w$ is obtained from $H$ by surgering along a subcollection of these arcs, and possibly along some honest compressing disks for $H_v$. In other words, $H_w$ is obtained from $H_v$ by a collection of honest and dishonest compressing disks. We may thus define $\eta_v(w)$ to be the barycenter of the simplex of $\plex{\CC \BB^r(H_v,\TT^1)}^{(1)}$ spanned by these compressing disks. Compatibility with the map $\iota$ then follows by construction. 
\end{proof}

\section{Stage 3: $\plex{\CC \BB^r(H,\TT^1)} \rightarrow \plex{\CC \EE(H,\TT^1)}$}
\label{s:Stage3}

\begin{thm}
\label{t:CE^r(H,T^1)}
Let $H$ be a surface in $M$ such that $\ind \plex{\CC \BB^r(H,\TT^1)}$ is well defined. Then either $\ind\plex{\CC \EE(H, \TT^1)} \le \ind \plex{\CC \BB^r(H,\TT^1)}$, or there is a collection $\DD$ of interior edge-compressing disks for $H$ such that 
\[\ind \plex{\CC \BB^r(H \slash \DD,\TT^1)} \le \ind \plex{\CC \BB^r(H,\TT^1)} -|\DD|+1.\]
\end{thm}

\begin{proof}
Since $\bdy H$ is normal, every element of $\BB^r(H,\TT^1)$ is isotopic to at least one element of $\EE(H,\TT^1)$. Let $X$ denote the subcomplex of $\EE(H,\TT^1)$ spanned by those disks that are isotopic to elements of $\BB^r(H,\TT^1)$, together with the complex $\plex{\CC(H,\TT^1)}$. Then $X$ is homotopy equivalent (in fact, homeomorphic) to $\plex{\CC \BB^r(H,\TT^1)}$. If we set $Y=\plex{\CC \EE(H,\TT^1)}$, then $X \subset Y$. We may thus apply Lemma \ref{l:generalization}. The conclusion of this lemma gives us two possibilities. The first is that $\ind \plex{\CC \EE(H,\TT^1)} \le  \ind \plex{\CC \BB^r(H,\TT^1)}$, and thus we immediately obtain the desired result. The second possibility is that there is a simplex $\tau$ spanned by edge-compressing disks incident to interior edges, such that the homotopy index of the complex $Z$ spanned by 
\[\{x \in \plex{\CC \BB^r(H,\TT^1)}|\forall y \in \tau, \ x\mbox{ is adjacent to }y\}\] 
is at most $\ind \plex{\CC \BB^r(H,\TT^1)}-{\rm dim}(\tau)$. 

Let $\DD$ denote a collection of disks representing the vertices of $\tau$ which is pairwise disjoint in the complement of $N(\TT^1)$. Then ${\rm dim}(\tau)=|\DD|-1$. Let $H'$ denote the surface in $M-N(\TT^1)$ obtained from $H-N(\TT^1)$ by simultaneous surgery along the disks in $\DD$. Each element of $Z$ can then be identified with a compressing or real $\bdy$-compressing disk for $H'$. We  claim that $H \slash \DD$ is well-defined, and thus $Z=\plex{\CC \BB^r(H \slash \DD,\TT^1)}$. The desired result follows. 

To show $H \slash \DD$ is well-defined, we must prove that for every interior edge $e$ of $\TT^1$, every component of $H' \cap \bdy N(e)$ is essential. If not, then such a boundary component that is innermost on $\bdy N(e)$ bounds a subdisk $C$ of $\bdy N(e)$ that is isotopic to a compressing disk for $H'$. But, as this disk is isotopic into $\bdy N(\TT^1)$, it can be made disjoint from every other disk in $Z$. Thus, $Z$ would be a contractible complex, a contradiction. 
\end{proof}

\section{Stage 4: $\plex{\CC \EE(H,\TT^1)} \rightarrow \plex{\CC \EE_{\TT^2}(H,\TT^1)}$}

\begin{dfn}
Suppose $E \in \CC\EE(H,\TT^1)$. We say $E'$ is a {\it shadow} of $E$ if $\bdy E'=\bdy E$, and $E'$ is disjoint from $\TT^2$ away from its boundary. (The interior of a shadow $E'$ may intersect $H$.)  We define the set $\CC \EE_{\TT^2}(H,\TT^1)$ to be the subset of $\CC\EE(H,\TT^1)$ consisting of those disks that have shadows. 
\end{dfn}

\begin{thm}{\rm [cf. \cite{TopIndexI}, Theorem 3.2.]}
\label{t:SecondTransitionTheorem}
Suppose the homotopy index of $\plex{\CC\EE(H,\TT^1)}$ is $n$. Then $H$ may be isotoped (rel $\TT^1$) so that 
	\begin{enumerate}
		\item $H$ meets the 2-simplices of $\TT^2$ in $p$ points of tangency, for some $p \le n$. Away from these tangencies $H$ is transverse to $\TT^2$. 
		\item The complex $\plex{\CC\EE_{\TT^2}(H,\TT^1)}$ has homotopy index $i \le n-p$.
	\end{enumerate}
\end{thm}

\begin{proof}
When $\ind \plex{\CC\EE(H,\TT^1)}=0$ the result is immediate, as $\plex{\CC\EE_{\TT^2}(H,\TT^1)} \subset \plex{\CC\EE(H,\TT^1)}$. We will assume, then, that $\ind \plex{\CC\EE(H,\TT^1)} = n \ge 1$. It follows that $\pi_{n-1}(\plex{\CC\EE(H,\TT^1)}) \ne 1$, and thus there is a map $\iota:S \to \plex{\CC\EE(H,\TT^1)}$ of an $(n-1)$-sphere $S$ into the $(n-1)$-skeleton of $\plex{\CC\EE(H,\TT^1)}$ which is not homotopic to a point. Let $B$ be the cone on $S$ to a point $z$. (The point $z$ is necessarily not in $\plex{\CC\EE(H,\TT^1)}$.) Hence, $B$ is an $n$-ball. 

Our first challenge is to define a continuous family of surfaces $H_x$ in $M$ isotopic to $H$, where $x \in B$.  Let $\Sigma$ be a triangulation of $S=\bdy B$ so that the map $\iota$ is simplicial. Let $\{v_i\}$ denote the set of vertices of $\plex{\CC\EE(H,\TT^1)}$ that are contained in $\iota(S)$. For each $i$ choose a representative $D_i$ from the equivalence class of disks represented by $v_i$ so that if $(v_i,v_j)$ is an edge of $\plex{\CC\EE(H,\TT^1)}$, then $D_i \cap D_j=\emptyset$.

For each $i$, let $N_i$ be a neighborhood of $D_i$ in $M-N(\TT^1)$. Assume $N_i$ has been chosen to be small enough so that $N_i \cap N_j=\emptyset$ whenever $(v_i,v_j)$ is an edge of $\plex{\CC\EE(H,\TT^1)}$. Let  $H_i(t)$ be a family of surfaces in $N_i$ so that $H_i(0)=H \cap N_i$, and $H_i(1)$ is obtained from $H \cap N_i$ by shrinking the disk $D_i$ until it meets $\TT^2$ in at most a subarc of an edge of $\TT^1$.

Extend $\Sigma$ to a triangulation $\Sigma'$ on $B$ by coning each simplex of $\Sigma$ to the point $z$. Suppose $\{D_0,...,D_{n-1}\}$ is the image of an $(n-1)$-simplex $\Delta$ of $\Sigma$ under the map $\iota$. We now identify the $n$-simplex of $\Sigma'$ which is the cone on $\Delta$ with the unit cube in $\mathbb R^n$. Label the axes of $\mathbb R^n$ with the variables $t_0, ..., t_{n-1}$. Place $z$ at the origin, and the vertex $v$ of $\Delta$ such that $\iota(v)=D_i$ at the point with $t_i=1$ and $t_j=0$ for all $j \ne i$. If $p$ is at the barycenter of a face $\sigma$ of $\Delta$ then place it at the vertex of the cube where the coordinates corresponding to the vertices of $\sigma$ are $1$ and the other coordinates are 0. We now linearly extend over the entire simplex to complete the identification with the cube. Now, if $x$ is in this $n$-simplex then $x$ has coordinates $(t_0(x),...,t_{n-1}(x))$. Let $H_x$ be the surface obtained from $H$ by replacing $H \cap N_i$ with the surface $H_i(t_i(x))$, for each $i$ between $0$ and $n-1$. See Figure \ref{f:HxDefn}. Repeating this for each $n$-simplex of $\Sigma'$ gives us the complete family of surfaces $H_x$. 

\begin{figure}
\psfrag{a}{$D_0$}
\psfrag{b}{$D_1$}
\psfrag{c}{$D_2$}
\psfrag{A}{$t_0$}
\psfrag{B}{$t_1$}
\psfrag{C}{$t_2$}
\psfrag{z}{$z$}
\psfrag{H}{$H$}
\psfrag{S}{$S$}
\begin{center}
\includegraphics[width=5 in]{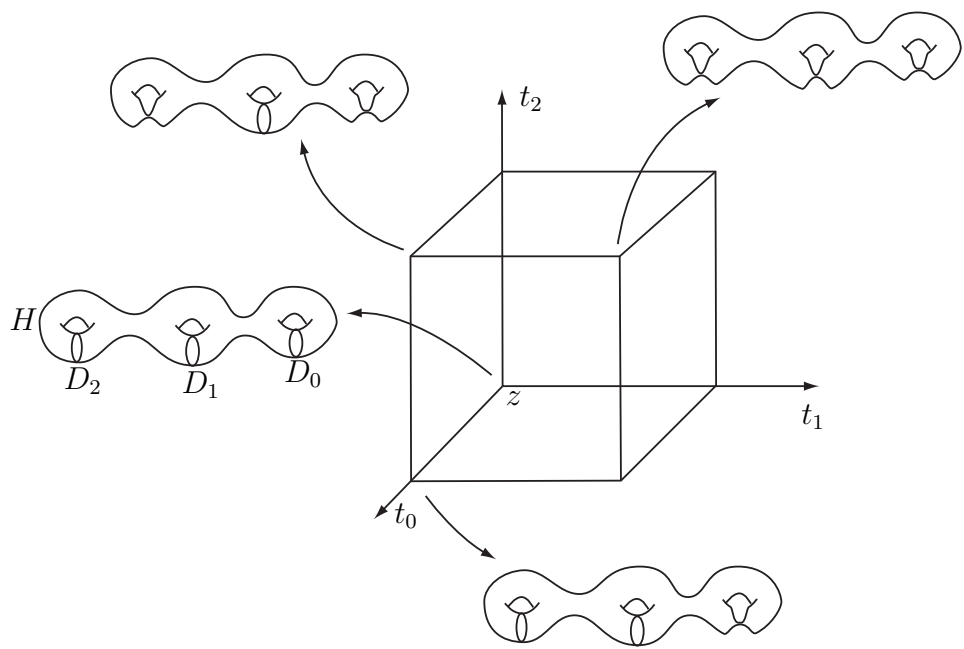}
\caption{A simplex $\Delta$ of $\Sigma'$, and a few of the surfaces $H_x$ for $x \in \Delta$. The union of the faces of the cube that do not meet $z$ is a simplex of $T$.}
\label{f:HxDefn}
\end{center}
\end{figure}

We assume $H$ is initially transverse to $\TT^2$. (That is, small perturbations of $H$ do not change $\plex{\CC\EE_{\TT^2}(H,\TT^1)}$.) For each $i$, the surface $H_i(t) \subset N_i$ is tangent to $\TT^2$ for finitely many values $\{t_i^j\}$ of $t$. Hence, for each $x \in B$ the surface $H_x$ is tangent to $\TT^2$ at finitely points, and each such point is in a distinct ball $N_i$. Note also that if $t_i(x)=t_i(y)$, then $H_x$ and $H_y$ agree inside of $N_i$. Hence, if $H_x$ is tangent to $\TT^2$ in $N_i$ then the surface $H_y$ will also be tangent to $\TT^2$, for all $y$ in the plane where $t_i(y)=t_i(x)$. It follows that each $n$-simplex of $\Sigma'$ is cubed by the points $x$ where $H_x$ is tangent to $\TT^2$. See Figure \ref{f:Cubing}. Hence, $B$ is cubed by the $n$-simplices of $\Sigma'$, together with this cubing of each such simplex. We denote this cubing of $B$ as $\Omega$. It follows that if $x$ is in a codimension $p$ cell of $\Omega$ then the surface $H_x$ is tangent to $\TT^2$ in at most $p$ points. 

\begin{figure}
\psfrag{A}{$t_0$}
\psfrag{B}{$t_1$}
\psfrag{C}{$t_2$}
\psfrag{t}{$\{t_0^j\}$}
\begin{center}
\includegraphics[width=3 in]{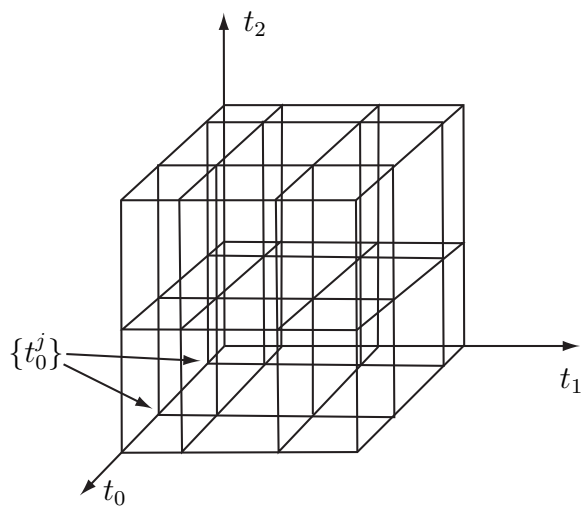}
\caption{A simplex $\Delta$ of $\Sigma'$ is cut up by planes into subcubes. Each such plane is determined by the points $x$ in which $H_x$ is tangent to $\TT^2$ in $N_i$, for some $i$.}
\label{f:Cubing}
\end{center}
\end{figure}

We now produce a contradiction by defining a continuous map $\Psi$ from $B$ into $\plex{\CC\EE(H,\TT^1)}$ which maps $S=\bdy B$ onto $\iota(S)$. The map $\Psi |S$ will be equal to $\iota$ on the barycenters of the $(n-1)$-cells of $\Sigma$, which will in turn imply that $\Psi$ maps $S$ onto $\iota(S)$ with the same degree as $\iota$. A contradiction follows as $\iota(S)$ is not homotopic to a point. 

For each $x \in B$ let $V_x=\plex{\CC\EE_{\TT^2}(H_x,\TT^1)}$. If $\tau$ is a cell of $\Omega$, then we define $V_\tau$ to be the set $V_x$, for any choice of $x$ in the interior of $\tau$. Note that if $x$ and $y$ are in the interior of the same cell $\tau$ of $\Omega$, then the pair $(H_x,\TT^2)$ is isotopic to $(H_y,\TT^2)$. Hence, $V_x=V_y$, and thus $V_\tau$ is well defined. The map $\Psi$ defined below will take each cell $\tau$ of $\Omega$ into $V_\tau$. First, we establish a few properties of $V_\tau$. 

\begin{clm}
\label{c:subset}
Suppose $\sigma$ is a cell of $\Omega$ which lies on the boundary of a cell $\tau$. Then $V_\sigma \subset V_\tau$. 
\end{clm}

\begin{proof}
Pick $x \in \sigma$ and $y \in \tau$. If $D \in V_x$ then $D$ is some kind of compression for $H_x$ that can be isotoped to be disjoint from $\TT^2$.  To show $D \in V_y$ we must show that $D$ is some kind of compression for $H_y$ disjoint from $\TT^2$. Note that $H_y \cap \TT^2$ is obtained from $H_x \cap \TT^2$ by resolving some tangency. Hence, any loop or arc of $H_x \setminus \TT^2$ is isotopic to a loop or arc of $H_y \setminus \TT^2$. It follows that since $D \cap H_x$  was a collection of loops and arcs on $H_x$ disjoint from $\TT^2$, then $D \cap H_y$ will be a similar such collection. Hence, $D \in \plex{\CC\EE_{\TT^2}(H_y,\TT^1)}=V_y$. 
\end{proof}

\begin{clm}
\label{c:contradiction}
For each cell $\tau$ of $\Omega$, 
\[\pi_{i}(V_\tau) = 1 \mbox{ for all }i \le {\rm dim}(\tau)-1.\]
\end{clm}

\begin{proof}
Let $x$ be in the interior of a codimension $p$ cell $\tau$ of $\Omega$. Then the dimension ${\rm dim}(\tau)$  is $n-p$. The surface $H_x$ is tangent to $\TT^2$ in at most $p$ points, and is transverse to $H_x$ elsewhere. Recall $V_x=\plex{\CC\EE_{\TT^2}(H_x,\TT^1)}$. Thus, if the theorem is false then $V_x$  is non-empty, and $\pi_i(V_x)=1$ for all $i \le n-p-1={\rm dim}(\tau)-1$. 
\end{proof}

We now define $\Psi$ on the $0$-skeleton of $\Omega$. For each $0$-cell $x \in \Omega$, we will choose a point in $V_x$ to be $\Psi(x)$. If $x$ is in the interior of $B$ then $\Psi(x)$ may be chosen to be an arbitrary point of $V_x$. If $x$ is a point of $S=\bdy B$ then $x$ is contained in (perhaps more than one) $(n-1)$-simplex $\Delta_x$ of $\Sigma$. Let $\Delta'_x$ denote the face of $\Delta_x$ spanned by the vertices $v$ such that $t_i(v)=1$ if $t_i(x)=1$, and $t_i(v)=0$ otherwise. (Note that if $x$ was on the boundary of $\Delta_x$, so that it was also contained in some other $(n-1)$-simplex of $\Sigma$, then we still end up with the same simplex $\Delta_x'$ of $\Sigma$.) So, for example, if $x$ is at the barycenter of $\Delta_x$ then $\Delta '_x=\Delta _x$. By construction, for each vertex $v$ of $\Delta'_x$ the surface $H_x$ has a compression of some kind $D$ disjoint from $\TT^2$. Hence, for all $y$ near $x$ the disk $D$ is a compression of some kind for $H_y$ that is disjoint from $\TT^2$. It follows that the entire simplex $\iota(\Delta'_x)$ is contained in $V_x$, and thus we may choose the barycenter of $\iota(\Delta'_x)$ to be the image of $\Psi(x)$. In particular, if $x$ is the barycenter of $\Delta_x$ then $\Psi(x)=\iota(x)$. 

We now proceed to define the rest of the map $\Psi$ by induction. Let $\tau$ be a $d$-dimensional cell of $\Omega$. By induction,  assume $\Psi$ has been defined on the $(d-1)$-skeleton of $\Omega$. In particular, $\Psi$ has been defined on $\bdy \tau$. Suppose $\sigma$ is a face of $\tau$.  By Claim \ref{c:subset} $V_\sigma \subset V_\tau$. By assumption $\Psi|\sigma$ is defined and $\Psi(\sigma) \subset V_\sigma$. We conclude $\Psi(\sigma) \subset V_\tau$ for all $\sigma \subset \bdy \tau$, and thus
\begin{equation}
\label{e:boundary}
\Psi(\bdy \tau) \subset V_\tau.
\end{equation}
Since $d={\rm dim}(\tau)$ it follows from Claim \ref{c:contradiction} that $\pi_{(d-1)}(V_\tau)=1$. Since $d-1$ is the dimension of $\bdy \tau$, we can thus extend $\Psi$ to a map from $\tau$ into $V_\tau$. 

What remains to be shown is that if $\tau$ is in $S=\bdy B$ then the extension of $\Psi$ from $\bdy \tau$ to $\tau$ may be done so that $\Psi(\tau) \subset \iota(S)$. Let $\Delta_\tau$ be the simplex of $\Sigma$ whose interior contains $\tau$.
We need only show that $\Psi(\bdy \tau) \subset V_\tau \cap \iota(\Delta _\tau)$. Since $V_\tau \cap \iota(\Delta _\tau)$ will be a subsimplex of $\iota(\Delta_\tau)$, it follows that $\Psi$ can be extended over $\tau$ to this subsimplex. 

By Equation \ref{e:boundary}, $\Psi(\bdy \tau) \subset V_\tau$. So all we must do now is to show $\Psi(\bdy \tau) \subset \iota(\Delta _\tau)$. Let $\sigma$ denote a face of $\tau$, and $\Delta _\sigma$ the simplex of $\Sigma$ whose interior contains $\sigma$. Then $\Delta_\sigma$ is contained in $\Delta _\tau$. By induction we may assume $\Psi(\sigma) \subset \iota(\Delta _\sigma)$. Putting this together we conclude $\Psi(\sigma) \subset \iota(\Delta _\tau)$ for each $\sigma \subset \bdy \tau$, and thus $\Psi(\bdy \tau) \subset \iota(\Delta _\tau)$. 
\end{proof}

\section{Stage 5: $\plex{\CC\EE_{\TT^2}(H,\TT^1)} \rightarrow \plex{\CC\EE(H,\TT^2)}$}

\begin{dfn}
Let $M^1=M-N(\TT^1)$ and $H^1=H \cap M^1$.
\end{dfn}

\begin{lem}
\label{l:NonCompressionEffect}{\rm [cf. \cite{TopIndexI}, Lemma 3.6.]}
Suppose $\plex{\CC\EE_{\TT^2}(H,\TT^1)}$ has well-defined homotopy index.  Let $D \in \CC\EE(H,\TT^2) - \CC\EE(H,\TT^1)$. Then $\CC\EE_{\TT^2}(H/D,\TT^1)=\CC\EE_{\TT^2}(H,\TT^1)$.
\end{lem}

\begin{proof}
Note that $\EE(H,\TT^2) \subset \EE_{T^2}(H,\TT^1) \subset \EE (H,\TT^1)$. Thus, the disk $D$ must be an element of $\CC(H,\TT^2)$ that is not in $\CC(H,\TT^1)$. We conclude $\bdy D$ cuts off a subdisk $\mathcal D \subset H^1$ that meets $\TT^2$.

Let $M(H,\TT^2)$ and $N(D)$ be as given in Definition \ref{d:Compression}. Then $H/D$ is obtained from $H$ by removing $N(D) \cap H$ from $H$ and 
replacing it with the frontier $D_*$ of $N(D)$ in $M(H,\TT^2)$ (followed by removing any resulting disk or sphere components). 

\begin{clm}
$\CC\EE_{\TT^2}(H/D,\TT^1) \subset \CC\EE_{\TT^2}(H,\TT^1)$.
\end{clm}

\begin{proof}
Suppose $E \in \CC\EE_{\TT^2}(H/D,\TT^1)$. Then $\bdy E$ can be isotoped off of $D_*$. If $E$ now meets the ball $N(D)$ then it can be further isotoped so that $E \cap N(D)$ is a collection of disks parallel to $D$. But then each component of $E \cap N(D)$ can be swapped with a disk parallel to $\mathcal D$. The resulting disk has the same boundary as $E$, but is disjoint from $H$. By the irreducibility of $M^1$ this disk must therefore be properly isotopic to $E$. See Figure \ref{f:HtoH/Dpart2}. We conclude that $E$ was a compressing or edge-compressing disk for $H^1$ that persisted as a compressing or edge-compressing disk for $H^1/D$. $E$ is therefore a compressing or edge-compressing disk for $H^1$ that is disjoint from $D$.

\begin{figure}
\psfrag{D}{$D$}
\psfrag{1}{$\mathcal D$}
\psfrag{d}{$D_*$}
\psfrag{H}{$H$}
\psfrag{h}{$H/D$}
\psfrag{F}{$\TT^2$}
\psfrag{E}{$E$}
\psfrag{e}{$E'$}
\begin{center}
\includegraphics[width=4.5 in]{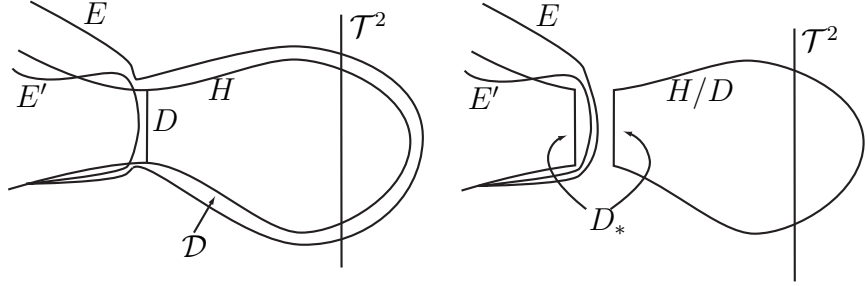}
\caption{Since $D$ is not a compressing disk for $H^1$, any compressing disk $E$ for $H^1/D$ (right figure) is always isotopic to a compressing disk for $H^1$ (left figure). If $E'$ is a shadow for $E$ as a compressing disk for $H^1/D$ (right figure), then $E'$ is a shadow for $E$ as a compressing disk for $H^1$ (left figure).}
\label{f:HtoH/Dpart2}
\end{center}
\end{figure}

Now let $E'$ be a shadow for $E$ as a compressing or edge-compressing disk for $H^1/D$. As $\bdy E'=\bdy E$, it follows that $\bdy E' \cap D_*=\emptyset$. So, if $E'$ meets the ball $N(D)$, then it meets it in disks parallel to $D$. The disk $E'$ thus meets $H$ in loops/arcs  isotopic to $E' \cap H/D$, together with loops/arcs  parallel to $D \cap H$. It follows that the interior of $E'$ meets $H$ in inessential loops/arcs, and thus $E'$ is a shadow for $E$ as a compressing or edge-compressing disk for $H^1$, i.e.~$E \in \CC\EE_{\TT^2}(H,\TT^1)$. See Figure \ref{f:HtoH/Dpart2}.
\end{proof}

\begin{clm}
$\CC\EE_{\TT^2}(H,\TT^1) \subset \CC\EE_{\TT^2}(H/D,\TT^1)$.
\end{clm}

\begin{proof}
Let $E$ now denote an element of $\CC\EE_{\TT^2}(H,\TT^1)$. Thus, $\bdy E \cap \TT^2=\emptyset$. We assume $E$ has been chosen so that $|E \cap \mathcal D|$ is minimal. First we suppose $E \cap \mathcal D=\emptyset$. If the interior of $E$ meets $D$ then we may surger it off by a standard innermost disk argument. So in this case we may assume $E \cap D=\emptyset$. Since $E$ is a compressing or edge-compressing disk for $H^1$ but $D$ is not, it now follows that $E$ is a compressing or edge-compressing disk for $H^1/D$. Any shadow for $E$ as a compressing or edge-compressing disk for $H^1$ will be a shadow for $E$ as a compressing or edge-compressing disk for $H^1/D$, and thus $E \in \plex{\CC\EE_{\TT^2}(H/D,\TT^1)}$.

Finally, we consider the case $E \cap \mathcal D \ne \emptyset$. Our goal is to isotope $E$ to a compressing or edge-compressing disk $E_0 \in \plex{\CC\EE_{\TT^2}(H,\TT^1)}$ such that $|E_0 \cap \mathcal D|<|E \cap \mathcal D|$, contradicting our minimality assumption.

Let $\gamma$ denote an arc of $\bdy E \cap \mathcal D$ that is outermost on $\mathcal D$. Then $\gamma$ cuts a disk $\mathcal D'$ off of $\mathcal D$ whose interior does not meet $E$. We can use the disk $\mathcal D'$ to guide an isotopy of both $E$ and its shadow $E'$ to a compressing disk $E_0$ for $H$ and a disk $E_0'$ with $\bdy E_0 =\bdy E_0'$. See Figure \ref{f:EtoE0}. Note that $|E_0 \cap \mathcal D|<|E \cap \mathcal D|$. If $\mathcal D' \cap \TT^2=\emptyset$, then it follows from the fact that $E'$ was a shadow of $E$ that $E_0'$ will be a shadow of $E_0$. Thus, $E_0 \in \plex{\CC\EE_{\TT^2}(H,\TT^1)}$ as desired.

\begin{figure}
\psfrag{D}{$\mathcal D$}
\psfrag{C}{$D$}
\psfrag{d}{$\mathcal D'$}
\psfrag{H}{$H$}
\psfrag{F}{$E_0$}
\psfrag{f}{$E'_0$}
\psfrag{E}{$E$}
\psfrag{e}{$E'$}
\begin{center}
\includegraphics[width=4.5 in]{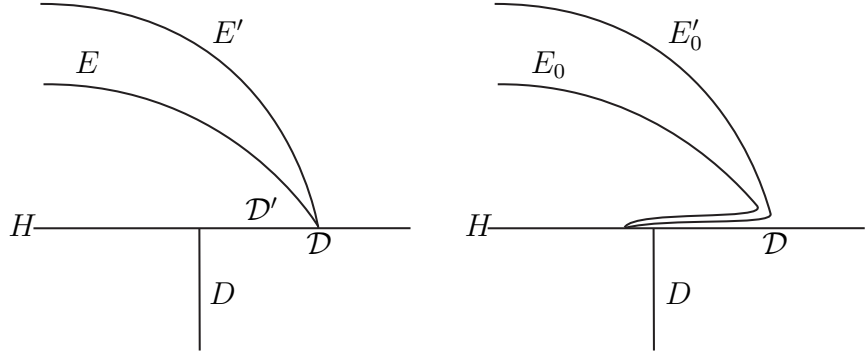}
\caption{Using the disk $\mathcal D'$ to obtain $E_0$ and $E_0'$ from $E$ and $E'$.}
\label{f:EtoE0}
\end{center}
\end{figure}

If $\mathcal D' \cap \TT^2 \ne \emptyset$ then the disk $E_0'$ will not be a shadow for $E_0$, since $E_0' \cap \TT^2 \ne \emptyset$. What remains then is to show that nonetheless, $E_0$ has a shadow.

Let $N(\TT^2)$ denote a small product neighborhood of $\TT^2$. Since $E_0' \cap \TT^2 \ne \emptyset$, it follows that $E_0' \cap \bdy N(\TT^2) \ne \emptyset$. Let $\delta$ denote a loop of $E_0' \cap \bdy N(\TT^2)$ that is outermost on $E_0'$. As $\TT^2$ is incompressible, $\delta$ bounds a subdisk $F_*$ of $\bdy N(\TT^2)$. See Figure \ref{f:E0toShadow}.

Although $\TT^2$ may not be transverse to $H$, the surface $\bdy N(\TT^2)$ will be. Thus, the disk $F_*$ meets $H$ in a collection of loops. We claim these loops are inessential on $H^1$, and thus $F_*$ can be used to surger $E_0'$ to a disk which meets $\TT^2$ fewer times. The new disk will meet $H$ more times, but each new intersection introduced will be inessential on $H^1$. Thus, by repeating this process we transform $E_0'$ to a shadow for $E_0$, as desired.

To obtain a contradiction, suppose at least one loop of $F_* \cap H$ is essential on $H^1$. Let $\alpha$ be a such loop that is innermost on $F_*$. The loop $\alpha$ bounds a subdisk $A'$ of $F_*$ whose interior may meet $H^1$ in inessential loops. See Figure \ref{f:E0toShadow}. We claim $A'$ is the shadow of a compressing disk $A$ for $H^1$, and thus $A \in \plex{\CC\EE_{\TT^2}(H,\TT^1)}$.

\begin{figure}
\psfrag{f}{$E_0'$}
\psfrag{C}{$D$}
\psfrag{D}{$\mathcal D$}
\psfrag{H}{$H$}
\psfrag{G}{$F_*$}
\psfrag{B}{$B$}
\psfrag{A}{$A'$}
\psfrag{d}{$\delta$}
\psfrag{a}{$\alpha$}
\psfrag{b}{$\beta$}
\psfrag{F}{$N(\TT^2)$}
\begin{center}
\includegraphics[width=4.5 in]{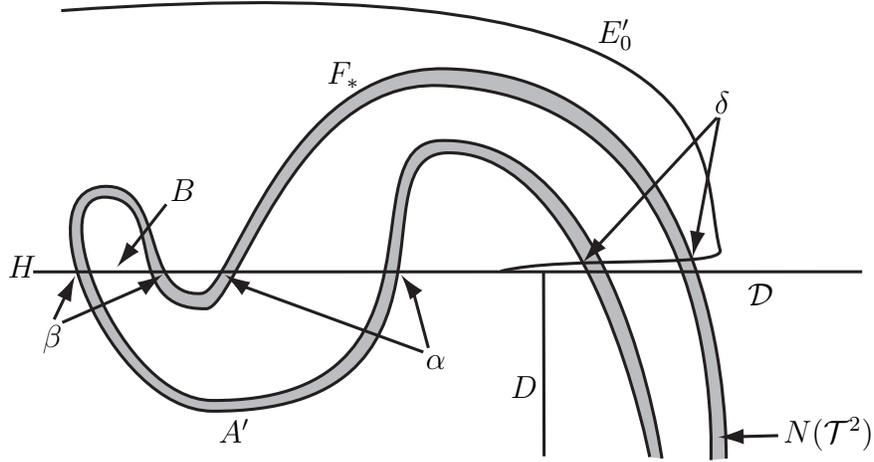}
\caption{The curves $\alpha$, $\beta$, and $\delta$, and the disks $A'$, $B$, and $F_*$.}
\label{f:E0toShadow}
\end{center}
\end{figure}

Let $\beta$ denote a loop of $A' \cap H$ that is innermost on $A'$. As $\beta$ is inessential on $H^1$ it bounds a subdisk $B$ of $H^1$. See Figure \ref{f:E0toShadow}. The disk $B$ can be used to surger $A'$, lowering $|A' \cap H|$. Continuing in this way we arrive at a disk $A$ with the same boundary as $A'$ but whose interior is disjoint from $H$. As $\bdy A'=\bdy A$ is essential on $H^1$, we conclude $A$ is a compressing disk for $H^1$. The disk $A'$ is then a shadow for $A$, and thus $A \in \plex{\CC\EE_{\TT^2}(H,\TT^1)}$.

Finally, suppose $X$ is any other element of $\plex{\CC\EE_{\TT^2}(H,\TT^1)}$. As $\bdy A =\bdy A' \subset F_* \subset \bdy N(\TT^2)$ and $\bdy X \cap \TT^2 =\emptyset$, it follows that $\bdy X \cap \bdy A=\emptyset$. By a standard innermost disk argument (and the irreducibility of $M^1$) we may isotope $X$ to remove any intersections of its interior with the interior of $A$. Thus, we may assume $A \cap X=\emptyset$. The disk $X$ is therefore connected to the disk $A$ by an edge in $\plex{\CC\EE_{\TT^2}(H,\TT^1)}$. As this holds for all disks $X \in \plex{\CC\EE_{\TT^2}(H,\TT^1)}$, we conclude $\plex{\CC\EE_{\TT^2}(H,\TT^1)}$ is contractible to $A$. As $\plex{\CC\EE_{\TT^2}(H,\TT^1)}$ is not contractible, we have reached a contradiction.
\end{proof}

The two claims complete the proof of Lemma \ref{l:NonCompressionEffect}.
\end{proof}

At the end of the previous section were were left with a surface $H$ with $\plex{\CC\EE_{\TT^2}(H,\TT^1)}$ having a well-defined homotopy index. In the next theorem we transition to a similar statement about $\plex{\CC\EE (H,\TT^2)}$.

\begin{thm}
\label{t:2skeleton}
Suppose $H$ is transverse to $\TT^2$ away from a finite number of points of tangency, and $\plex{\CC\EE_{\TT^2}(H,\TT^1)}$ has homotopy index $n$. Then $H$ may be isotoped so that any point of tangency is a saddle contained in $\TT^2$, and $\plex{\CC\EE(H,\TT^2)}$ has homotopy index $n$.
\end{thm}

\begin{proof}
Suppose $D \in \CC\EE(H,\TT^2) - \CC\EE(H,\TT^1)$. Then $\bdy D$ cuts off a subdisk $D'$ of $H^1$ that meets $\TT^2$. The surface $H/D$ is obtained by exchanging $D'$ with $D$. By the incompressibility of $\bdy M$ and the irreducibility of $M$, this surface is isotopic to $H$ in $M$. As  $H/D$ meets $\TT^2$ fewer times,  a similar surgery can only be performed a finite number of times. Thus, after a maximal sequence of such surgeries we obtain a surface isotopic to $H$ (which we continue to denote by the same letter), where every compressing or edge-compressing disk for $H-N(\TT^2)$ is a compressing or edge-compressing disk for $H^1$. It follows that there are no longer any center tangencies in $H \cap \TT^2$. Furthermore, by Lemma \ref{l:NonCompressionEffect} the complex $\plex{\CC\EE_{\TT^2}(H,\TT^1)}$ has remained unchanged. Thus, if this complex had homotopy index $n$ to begin with, then it still has homotopy index $n$.

Now suppose $D \in \plex{\CC\EE(H,\TT^2)}$. Then by the previous paragraph, $D \in \plex{\CC\EE(H,\TT^1)}$. Furthermore, as $D$ lies in the complement of $\TT^2$, $D \in \plex{\CC\EE_{\TT^2}(H,\TT^1)}$. We conclude $\plex{\CC\EE(H,\TT^2)} \subset \plex{\CC\EE_{\TT^2}(H,\TT^1)}$. We now claim  the opposite inclusion is true as well, and thus $\plex{\CC\EE(H,\TT^2)} = \plex{\CC\EE_{\TT^2}(H,\TT^1)}$.

Suppose $E \in \plex{\CC\EE_{\TT^2}(H,\TT^1)}$. Let $E'$ be a shadow of $E$. Let $\beta$ be a loop of $E' \cap H^1$ that is innermost on $E'$. Then $\beta$ bounds subdisks $C \subset E'$ and  $C' \subset H^1$. If $C' \cap \TT^2 \ne \emptyset$, then $C$ is a compressing disk for $H-N(\TT^2)$ that is not a compressing disk for $H^1$, a contradiction. We conclude $C' \cap \TT^2=\emptyset$. Since $E' \cap \TT^2=\emptyset$ and $C \subset E'$, we conclude $C \cap \TT^2=\emptyset$. The sphere $C \cup C'$ thus bounds a ball in the complement of $\TT^2$ that we can use to guide an isotopy of $C$ to $C'$. (This may remove other components of $E' \cap C$ as well.) We thus transform the disk $E'$ to a disk $E''$ such that $\bdy E''=\bdy E$, $E'' \cap \TT^2=\emptyset$, and $|E'' \cap H^1|<|E' \cap H^1|$. Continuing in this way we arrive at a compressing disk for $H^1$ with the same boundary as $E$, which is disjoint from $\TT^2$. Thus $E \in \plex{\CC\EE(H,\TT^2)}$, completing the proof that $\plex{\CC\EE_{\TT^2}(H,\TT^1)} \subset \plex{\CC\EE(H,\TT^2)}$
\end{proof}

\begin{lem}
\label{c:2skeleton}
Suppose $H$ is transverse to $\TT^2$ away from a finite number of saddle tangencies, and  $\plex{\CC\EE(H,\TT^2)}$ has well-defined homotopy index. Then $H$ is transverse to $\TT^2$, and meets each 2-simplex of $\TT^2$ in a collection of normal arcs.
\end{lem}

\begin{proof}
Suppose $\Delta$ is a 2-simplex of $\TT^2$. If $H \cap \Delta$ contains a loop, a non-normal arc, or a saddle tangency, then some subdisk $D$ of $\Delta$ that can be pushed off $\TT^2$ to form a compressing or edge-compressing disk for $H-N(\TT^2)$. We conclude $\plex{\CC\EE(H,\TT^2)}$ could not have been empty. As $D$ is isotopic into $\bdy (M-N(\TT^2))$, it can be made disjoint from any other element of $\plex{\CC\EE(H,\TT^2)}$. This implies every vertex of $\plex{\CC\EE(H,\TT^2)}$ is connected to $D$ by an edge, contradicting the assumption that $\plex{\CC\EE(H,\TT^2)}$ is not contractible.
\end{proof}

\section{From $H$ to the components of $H-N(\TT^2)$.}

\begin{thm}
\label{t:IndexSum}{\rm [cf. \cite{TopIndexI}, Theorem 4.7.]}
Suppose $\ind\plex{\CC \EE(H,\TT^2)} = n$. Then 
\[\sum \limits _{\Delta \in \TT^3} \ind\plex{\CC \EE(H \cap \Delta,\TT^2)} = n.\]
\end{thm}

\begin{proof}
Let $\{\Delta_i\}$ denote the 3-simplices in $\TT^3$. Let $H_i=H \cap \Delta_i^3$. Notice that elements of $\CC\EE(H_i, \TT^2)$ and $\CC\EE(H_j, \TT^2)$ are disjoint when $i \ne j$.  Hence, $\plex{\CC \EE(H, \TT^2)}$  is the join of all of the complexes $\plex{\CC\EE(H_i, \TT^2)}$. 

The proof is by induction on $m$, the number of tetrahedra in $\TT$. The base case $m=1$ is trivial.  Let $G=\bigcup \limits _{i=2} ^m H_i$, so that $\plex{\CC\EE(H, \TT^2)}$ is the join of $\plex{\CC\EE(H_1, \TT^2)}$ and $\plex{\CC\EE(G, \TT^2)}$.

We first observe that $\plex{\CC\EE(H_1, \TT^2)}$ and $\plex{\CC\EE(G, \TT^2)}$ are not contractible. This is because the join of a contractible space with any other space is also contractible. We conclude that is either complex is contractible then the complex $\plex{\CC \EE(H, \TT^2)}$ would have been contractible, a contradiction. Since neither is contractible, each has a well-defined homotopy index. 

If either $H_1$ or $G$ has local index 0 then the result is immediate, as the join of a complex with the empty set is the same complex. We assume, then, that the local index of $H_1$  is $n \ge 1$ and the local index of $G$  is $m \ge 1$.

By definition, $(n-1)$ is the smallest $i$ such that $\pi_i(\plex{\CC\EE(H_1, \TT^2)}) \ne 1$, and $(m-1)$ is the smallest $j$ such that  $\pi_j(\plex{\CC\EE(G, \TT^2)}) \ne 1$. Our goal is to show that $(n+m-1)$ is the smallest $k$ such that \[\pi_{k} (\plex{\CC\EE(H_1 \cup G, \TT^2)}) = \pi_{k} (\plex{\CC\EE(H_1, \TT^2)} * \plex{\CC\EE(G, \TT^2)}) \ne 1.\]

When $n=2$ then $\pi_1(\plex{\CC\EE(H_1, \TT^2)}) \ne 1$. Suppose $H_1$ separates $\Delta_1^3$ into $\VV$ and $\WW$. Let $\VV(H_1)$ and $\WW(H_1)$ denote the subsets of $\CC\EE(H_1, \TT^2)$ spanned by the compressing and edge-compressing disks that lie in $\VV$ and $\WW$, respectively. By an argument identical to the one given by McCullough in \cite{mccullough:91}, $\plex{\VV(H_1)}$ and $\plex{\WW(H_1)}$ are contractible. If we contract these to points $p_\VV$ and $p_\WW$, then the remaining 1-simplices of $\plex{\CC\EE(H_1, \TT^2)}$ join these two points. The fundamental group $\pi_1(\plex{\CC\EE(H_1, \TT^2)})$ is generated by these 1-simplices. The remaining 2-simplices have become bigons that run once over each of two 1-simplices. Hence, each such 2-simplex gives rise to a relation in $\pi_1(\plex{\CC\EE(H_1, \TT^2)})$ that kills one generator. It follows that $\pi_1(\plex{\CC\EE(H_1, \TT^2)})$ is free, and hence the non-triviality of $\pi_1(\plex{\CC\EE(H_1, \TT^2)})$ implies that the homology group $H_1(\plex{\CC\EE(H_1, \TT^2)})$ is also non-trivial. Similarly, if $m=2$ we conclude $H_1(\plex{\CC\EE(G, \TT^2)})$ is non-trivial. For $n \ge 3$ the non-triviality of $H_{n-1}(\plex{\CC\EE(H_1, \TT^2)})$ follows from the Hurewicz Theorem. 

By Lemma 2.1 from \cite{milnor}:
\begin{eqnarray*}
\tilde H_{n+m-1} (\plex{\CC\EE(H_1, \TT^2)} &*& \plex{\CC\EE(G, \TT^2)})\\
 & \cong & \sum \limits _{i+j=n+m-2} \tilde H_i(\plex{\CC\EE(H_1, \TT^2)}) \otimes \tilde H_j(\plex{\CC\EE(G, \TT^2)})\\
&& + \sum \limits _{i+j=n+m-3} {\rm Tor}(\tilde H_i(\plex{\CC\EE(H_1, \TT^2)}), \tilde H_j(\plex{\CC\EE(G, \TT^2)})).
\end{eqnarray*}

In particular, it follows from the fact that $(n-1)$ is the smallest $i$ such that $H_i(\plex{\CC\EE(H_1, \TT^2)})$ is non-trivial, and $(m-1)$ is the smallest $j$ such that  $H_j(\plex{\CC\EE(G, \TT^2)})$ is non-trivial, that $(n+m-1)$ is the smallest $k$ such that $H_{k} (\plex{\CC\EE(H_1, \TT^2)} * \plex{\CC\EE(G, \TT^2)})$ is non-trivial. 

\end{proof}

\bibliographystyle{alpha}

\begin{thebibliography}{BDTSb}

\bibitem[Baca]{nlizg}
D.~Bachman.
\newblock 2-normal surfaces.
\newblock Preprint. Available at {\tt http://arxiv.org/abs/math/0309437}.

\bibitem[Bacb]{TopMinNormalII}
D.~Bachman.
\newblock Normalizing {T}opologically {M}inimal {S}urfaces {I}{I}: {D}isks.
\newblock Preprint.

\bibitem[Bacc]{TopMinNormalIII}
D.~Bachman.
\newblock Normalizing {T}opologically {M}inimal {S}urfaces {I}{I}{I}: {I}ndex 2
  {S}urfaces.
\newblock Preprint.

\bibitem[Bacd]{Stabilizing}
David Bachman.
\newblock Stabilizing and destabilizing {H}eegaard splittings of sufficiently
  complicated 3-manifolds.
\newblock {\em Mathematische Annalen}, pages 1--32.
\newblock 10.1007/s00208-012-0802-4.

\bibitem[Bac02]{crit}
D.~Bachman.
\newblock Critical {H}eegaard surfaces.
\newblock {\em Trans. Amer. Math. Soc.}, 354(10):4015--4042 (electronic), 2002.

\bibitem[Bac10]{TopIndexI}
D.~Bachman.
\newblock Topological {I}ndex {T}heory for surfaces in 3-manifolds.
\newblock {\em Geometry \& Topology}, 14(1):585--609, 2010.

\bibitem[BDTSa]{Index1Normal}
D.~Bachman, R.~Derby-Talbot, and E.~Sedgwick.
\newblock Almost {N}ormal {S}urfaces with boundary.
\newblock Preprint. Available at {\tt http://arxiv.org/abs/1203.4632}.

\bibitem[BDTSb]{HeegaardDehn}
D.~Bachman, R.~Derby-Talbot, and E.~Sedgwick.
\newblock Most {D}ehn filings preserve the set of {H}eegaard splittings.
\newblock In preparation.

\bibitem[CG87]{cg:87}
A.~J. Casson and C.~McA. Gordon.
\newblock Reducing {H}eegaard splittings.
\newblock {\em Topology and its Applications}, 27:275--283, 1987.

\bibitem[CM04a]{cm1}
Tobias~H. Colding and William~P. Minicozzi, II.
\newblock The space of embedded minimal surfaces of fixed genus in a
  3-manifold. {I}. {E}stimates off the axis for disks.
\newblock {\em Ann. of Math. (2)}, 160(1):27--68, 2004.

\bibitem[CM04b]{cm2}
Tobias~H. Colding and William~P. Minicozzi, II.
\newblock The space of embedded minimal surfaces of fixed genus in a
  3-manifold. {II}. {M}ulti-valued graphs in disks.
\newblock {\em Ann. of Math. (2)}, 160(1):69--92, 2004.

\bibitem[CM04c]{cm3}
Tobias~H. Colding and William~P. Minicozzi, II.
\newblock The space of embedded minimal surfaces of fixed genus in a
  3-manifold. {III}. {P}lanar domains.
\newblock {\em Ann. of Math. (2)}, 160(2):523--572, 2004.

\bibitem[CM04d]{cm4}
Tobias~H. Colding and William~P. Minicozzi, II.
\newblock The space of embedded minimal surfaces of fixed genus in a
  3-manifold. {IV}. {L}ocally simply connected.
\newblock {\em Ann. of Math. (2)}, 160(2):573--615, 2004.

\bibitem[Hak61]{haken:61}
W.~Haken.
\newblock Theorie der {N}orm\"alflachen.
\newblock {\em Acta Math.}, 105:245--375, 1961.

\bibitem[Hak68]{haken:68}
W.~Haken.
\newblock {\em Some Results on Surfaces in 3-Manifolds}.
\newblock M.A.A., Prentice Hall, 1968.

\bibitem[Hem01]{hempel:01}
J.~Hempel.
\newblock 3-manifolds as viewed from the curve complex.
\newblock {\em Topology}, 40:631--657, 2001.

\bibitem[Joh]{johnson3}
J.~Johnson.
\newblock Calculating {I}sotopy {C}lasses of {H}eegaard {S}plittings.
\newblock Preprint. Available at {\tt http://arxiv.org/abs/1004.4669}.

\bibitem[Kne29]{kneser:29}
H.~Kneser.
\newblock Geschlossene {F}l\"achen in driedimensionalen {M}annigfaltigkeiten.
\newblock {\em Jahresbericht der Dent. Math. Verein}, 28:248--260, 1929.

\bibitem[McC91]{mccullough:91}
Darryl McCullough.
\newblock Virtually geometrically finite mapping class groups of
  {$3$}-manifolds.
\newblock {\em J. Differential Geom.}, 33(1):1--65, 1991.

\bibitem[Mil68]{milnor}
J.~Milnor.
\newblock {\em Morse Theory}.
\newblock PUP, Princeton, New Jersey, 1968.

\bibitem[Rub95]{rubinstein:93}
J.~H. Rubinstein.
\newblock Polyhedral minimal surfaces, {H}eegaard splittings and decision
  problems for 3-dimensional manifolds.
\newblock In {\em Proceedings of the Georgia Topology Conference}, pages 1--20,
  1995.

\bibitem[Sto00]{stocking:96}
M.~Stocking.
\newblock Almost normal surfaces in 3-manifolds.
\newblock {\em Trans. Amer. Math. Soc.}, 352:171--207, 2000.

\end{thebibliography}

\end{document}